\documentclass[10pt, a4paper, reqno]{amsart}

\usepackage{amssymb, amsmath, amsthm}
\usepackage{enumitem}
\usepackage[all]{xy}
\usepackage{mathtools}
\usepackage{extarrows}
\usepackage[normalem]{ulem}
\usepackage{stackrel}
\usepackage{tikz}
\usepackage[english]{babel}

\usepackage[colorlinks=true, pdfstartview=FitV, linkcolor=blue, citecolor=blue, urlcolor=blue]{hyperref}

\colorlet{lightlightgray}{gray!20}
\colorlet{light1.5gray}{gray!35}

\newtheorem{theorem}{Theorem}[section]
\newtheorem*{theorem*}{Theorem}

\newtheorem*{maintheorem*}{Main Theorem}
\newtheorem{lemma}[theorem]{Lemma}
\newtheorem{lemmaAndDefinition}[theorem]{Lemma and Definition}
\newtheorem{corollary}[theorem]{Corollary}
\newtheorem{proposition}[theorem]{Proposition}

\newtheorem*{question*}{Question}

\theoremstyle{remark}
\newtheorem{remark}[theorem]{Remark}
\newtheorem{example}[theorem]{Example}

\theoremstyle{definition}
\newtheorem{definition}[theorem]{Definition}

\setlist[1]{labelindent=\parindent, leftmargin=*}

\DeclareMathOperator{\pr}{pr}
\DeclareMathOperator{\Aut}{Aut}

\DeclareMathOperator{\Bir}{Bir}
\DeclareMathOperator{\id}{id}

\DeclareMathOperator{\Spec}{Spec}

\DeclareMathOperator{\dom}{dom}

\DeclareMathOperator{\Mor}{Mor}

\DeclareMathOperator{\car}{char}

\DeclareMathOperator{\Hilb}{Hilb}

\DeclareMathOperator{\lociso}{lociso}
\DeclareMathOperator{\fib}{-fib}

\def\dashmapsto{\mapstochar\dashrightarrow}

\newcommand{\OO}{\mathcal{O}}

\newcommand{\QQ}{\mathbb{Q}}
\newcommand{\PP}{\mathbb{P}}

\newcommand{\NN}{\mathbb{N}}
\renewcommand{\AA}{\mathbb{A}}

\newcommand{\kk}{\textbf{k}}

\newcommand{\set}[2]{\left\{\,#1 \ | \ #2\,\right\}}
\newcommand{\Bigset}[2]{\left\{\,#1 \ \Big| \ #2\,\right\}}

\title[Algebraic Families of Birational Transformations]
{The Structure of Algebraic Families of Birational Transformations}
\author[A. Regeta, C. Urech, \and I. van Santen]
{Andriy Regeta, Christian Urech, \and Immanuel van Santen}
\address{\noindent Institut f\"{u}r Mathematik, Friedrich-Schiller-Universit\"{a}t Jena, \newline
	\indent Ernst-Abbe-Platz 2, DE-07737, Germany}
\email{andriyregeta@gmail.com}

\address{Departement Mathematik, 
	ETH Zürich,\newline
	\indent Rämistrasse 101, CH-8092 Zürich, Switzerland}
\email{christian.urech@gmail.com}

\address{Departement Mathematik und Informatik, 
	Universit\"at Basel,\newline
	\indent Spiegelgasse 1, CH-4051 Basel, Switzerland}
\email{immanuel.van.santen@math.ch}

\newcounter{claim}[theorem]
\renewcommand{\theclaim}{\arabic{claim}}

\newenvironment{claim}{\vspace{1.75\medskipamount} 
\refstepcounter{claim}\par\noindent\textbf{Claim \theclaim.}}{}

\begin{document}


\setcounter{tocdepth}{1}

\begin{abstract}
	We give a description of the algebraic families of birational transformations 
	of an algebraic variety $X$. As an application,
	we show that the morphisms to $\Bir(X)$ given by algebraic families
	satisfy a Chevalley type result and a certain fibre-dimension formula.
	Moreover, we show that the algebraic subgroups of $\Bir(X)$ are exactly
	the closed finite-dimensional subgroups with finitely many components.
	We also study algebraic families of birational transformations preserving
	a fibration.
	This builds on previous work of Blanc-Furter \cite{BlFu2013Topologies-and-str},
	Hanamura~\cite{Ha1987On-the-birational-}, and Ramanujam~\cite{Ra1964A-note-on-automorp}.
\end{abstract}

\maketitle

\tableofcontents

\section{Introduction}

To an irreducible algebraic variety $X$ over an algebraically closed field $\kk$ one associates its group of birational transformations $\Bir(X)$. In order to better understand and study $\Bir(X)$, it is of great interest to equip $\Bir(X)$ with additional algebraic structures that reflect families of birational transformations. In 
\cite{BlFu2013Topologies-and-str}, the longstanding open question whether $\Bir(\mathbb{P}^n)$ has the structure of an ind-group is answered negatively; this is in contrast to $\Aut(X)$, which is a group scheme if $X$ is projective \cite{MaOo1967Representability-o} and an ind-group if $X$ is affine \cite{FuKr2018On-the-geometry-of}. 

On the other hand, in 
\cite{BlFu2013Topologies-and-str}, the authors show several useful results about the algebraic and topological structure of  $\Bir(\mathbb{P}^n)$. In particular, they prove that the algebraic structure can be described by a countable family of varieties that is in a certain way universal. However, their methods only apply to $\Bir(\mathbb{P}^n)$ and the structure of $\Bir(X)$ for arbitrary $X$ remained poorly understood (see \cite{Bl2017Algebraic-structur} for a survey). The goal of this article is to change this. 

We give a suitable description of the algebraic structure of $\Bir(X)$ and we develop various insightful tools to study it, which in practice turn out to be as useful as the structure of an ind-group. Indeed, in the forthcoming article \cite{ReUrSa2024Groups-of-Biration}, the algebraic structure of $\Bir(X)$ will be exploited to show that the variety $\PP^n$ is uniquely determined up to birational equivalence among all varieties by the abstract group structure of $\Bir(\PP^n)$, as well as to show that all Borel subgroups of maximal solvable length $2n$ in $\Bir(\PP^n)$ are conjugate. 

An \emph{algebraic family} of birational transformations of $X$ parametrized by a variety $V$ is a $V$-birational map
$\theta \colon V \times X \dashrightarrow V \times X$
inducing an isomorphism between open dense subsets of $V\times X$ that both surject onto $V$ under the first projection. This induces a map 
$\rho_{\theta} \colon V \to \Bir(X)$ given by
$v \mapsto (x \dashmapsto \theta(v, x))$, which is called a \emph{morphism}
and whose image is called an \emph{algebraic subset} of $\Bir(X)$.
The \emph{Zariski topology} on $\Bir(X)$ is the finest topology such that all the morphisms are continuous. This point of view was first introduced by D\'emazure~\cite{De1970Sous-groupes-algeb} and further discussed by Serre \cite{Se2009Le-groupe-de-Cremo} and Blanc and Furter \cite{BlFu2013Topologies-and-str}. In another direction, Hanamura studied in \cite{Ha1987On-the-birational-} the algebraic structure on $\Bir(X)$ given by considering only flat families of birational transformations. In this case, $\Bir(X)$ can be identified with an open subset of the Hilbert scheme of $X\times X$. However, considering only flat families is very restrictive and not compatible with the group structure. In the present article, we will always work with the Zariski topology and the algebraic structure defined by D\'emazure. 

\medskip

In a first step, we prove the following result, which generalizes the description of the Zariski topology of $\Bir(\PP^n)$ and of all morphisms
to $\Bir(\PP^n)$ given in  \cite[Section~2]{BlFu2013Topologies-and-str} (see also \cite{HaMo2024Open-loci-of-ideal}) to arbitrary varieties $X$:

\begin{theorem}[Lemma~\ref{Lem.Family_factorizing_through_H_d},
	Corollary~\ref{Cor.Inductive-limit-topology}]
	\label{Thm.Main}
	Let $X$ be an irreducible variety. There exists a countable sequence of
	varieties $H_d$ and morphisms $\pi_d\colon H_d\to\Bir(X)$ for $d\geq 1$ such that the following is satisfied:
	\begin{enumerate}[left=0pt]
		\item \label{Main.Thm2} The morphisms $\pi_d$ are closed maps
		and the Zariski topology on $\Bir(X)$ 
		is  the inductive-limit topology with respect to the filtration 
		by the closed algebraic subsets 
		$\pi_1(H_1)\subseteq \pi_2(H_2)\subseteq \cdots \subseteq \Bir(X)$.
		\item \label{Main.Thm3} Let $V$ be a variety and $\rho\colon V\to\Bir(X)$ be a morphism. Then there exists an 
		open covering $(V_i)_{i\in I}$ of $V$ such that for each $i$ the restriction of $\rho$ to $V_i$ factors through a morphism of varieties $V_i\to H_{d_i}$ for some $d_i\geq 1$.
	\end{enumerate}
\end{theorem}

Point~\eqref{Main.Thm3} shows that all morphisms to $\Bir(X)$ can be recovered from the countably many morphisms $\pi_d$.
Hence, this structure is essentially 
as powerful as the one given by an ind-group
(Remark~\ref{Rem.Similar_as_ind-group}).
Theorem~\ref{Thm.Main} 
is the starting point to show a series of results, which were, to the best of our knowledge, open up to now.
For instance, we show that a Chevalley type result holds for morphisms:

\begin{theorem}[Corollary~\ref{Cor.Constructible_images}]\label{thm:chevalley}
	The image of a {constructible subset under a} 
	morphism to $\Bir(X)$ is {again} constructible.
\end{theorem}

One of the main ingredients to show Theorem~\ref{thm:chevalley} is the observation that for any closed irreducible algebraic subset $Z$
of $\Bir(X)$ there exists a morphism to $\Bir(X)$ with image 
in $Z$ that induces
a homeomorphism onto an open dense subset of $Z$
(Proposition~\ref{Prop.Nice_Parametrization_closed_alg_subset}). 
The proof uses Hanamura's description of $\Bir(X)$ by the Hilbert scheme (see Section~\ref{Sec.Morphisms_Hilbert}).

If $G$ is an algebraic group and $\rho\colon G\to\Bir(X)$ a morphism that is also a group homomorphism, we call the image of $\rho$ an \emph{algebraic subgroup} of $\Bir(X)$. We will see that algebraic subgroups are always closed
(see Corollary~\ref{Cor.algebraic_subset_group_is_closed}). 
On the other hand, we prove:

\begin{theorem}[Proposition~\ref{Prop.Ramanujam_gen}]
	Let $G\subset\Bir(X)$ be a finite-dimensional, closed, connected subgroup.
	Then $G$ has a unique structure of an algebraic group.
\end{theorem}

We define the \emph{dimension} of a subset $S\subset\Bir(X)$, 
following Ramanujam, as the supremum over all $d$ such that there exists an injective morphism $V\to\Bir(X)$ of a variety $V$ of dimension $d$ with image in $S$. We show that this definition coincides with the topological Krull dimension
for closed algebraic subsets:

\begin{theorem}[Corollary~\ref{Cor.Krull_dimensionII}]
	Let $Z\subset\Bir(X)$ be a closed algebraic subset. Then its dimension 
	is the maximal length of a strictly descending chain of irreducible closed subsets of $Z$. 
\end{theorem}

Furthermore, we prove the following result about the fibre-dimension of morphisms:

\begin{theorem}[Corollary~\ref{Cor.Fibre-dimension-formula}]
	Let $\rho \colon V \to \Bir(X)$ be a morphism
	with irreducible $V$.
	Then there is an open dense subset $U \subseteq V$ such that
	\[
	\dim_u (U \cap \rho^{-1}(\rho(u))) = 
	\dim V - \dim \overline{\rho(V)}
	\quad \quad \textrm{for all $u \in U$} \, ,
	\]
	where $\dim_u$ denotes the local dimension at $u$.
\end{theorem}

In Section~\ref{Sec.fibration}, we consider the following situation. Let $\pi\colon X\to Y$ be a dominant morphism with integral geometric generic fibre. Denote by $\Bir(X, \pi)\subset\Bir(X)$ the subgroup of birational transformations that induce a birational transformation on the base $Y$ and by $\Bir(X/Y)$ the subgroup of $\pi$-invariant transformations. Then $\Bir(X/Y)$ acts  by birational transformations on the geometric generic fibre 
$X_K$ of $\pi$, where $K$ is the algebraic closure of the 
function field $\kk(Y)$. 
We therefore obtain a homomorphism $\Bir(X/Y)\to\Bir(X_{K})$.

\begin{theorem}[Proposition~\ref{Prop.Birational_maps_preserving_a_fibration},
	\ref{Prop.Pull-back_continuous}]
	The homomorphisms $\Bir(X, \pi)\to\Bir(Y)$ and $\Bir(X/Y)\to\Bir(X_K)$ are continuous. 
\end{theorem}
In fact, we prove slightly more: if $\rho \colon V \to \Bir(X)$
is a morphism, then the composition with $\Bir(X, \pi) \to \Bir(Y)$
yields a morphism to $\Bir(Y)$.

\subsection*{A functorial approach?}
In this article, we focused on developing practical tools  to study $\Bir(X)$ and its algebraic structure. However, we hope that our results can serve as a first step towards a more conceptual study of the algebraic structure of $\Bir(X)$. 

	The definition of a morphism to $\Bir(X)$ extends also to a scheme $X$ over a fixed base scheme $S$
	(for the definition of a birational transformation in this context, the reader may consult 
	e.g.~\cite[\S9.7]{GoWe0Algebraic-geometryI}):
	Let $V$ be an $S$-scheme. An \emph{algebraic family  of birational transformations of $X$
		parametrized by $V$} is a birational transformation $\theta \colon V \times_S X \dasharrow V \times_S X$
		that induces an isomorphism $U_1 \xrightarrow{\sim} U_2$ 
		on schematically dense open subsets $U_1, U_2$ of $V \times_S X$
	such that for every $S$-morphism $V' \to V$ the pull-backs 
	\[
		V' \times_V U_1 \quad \textrm{and} \quad 
		V' \times_V U_2
	\]
	are schematically dense in $V' \times_V V \times_S X = V' \times_S X$.
	Then we have
	an associated map from the $S$-morphisms $\Mor_S(V', V)$ to the group of birational 
	transformations  $\Bir_{V'}(V' \times_S X)$ of $V' \times_S X$ over $V'$ 
	that defines a natural transformation between
	the contravariant functors $V' \mapsto \Mor_S(V', V)$ and 
	$V' \mapsto \Bir_{V'}(V' \times_S X)$,
	which we call a \emph{morphism}. This approach potentially leads to a systematic treatment
	of the group of birational transformations of $X$ by generalizing our results to this setting.
	In this paper, we do not adopt this point of view. However, it could be an interesting path to pursue, and it would be desirable to at least generalize our results to the setting of arbitrary fields within this functorial approach.

\subsection*{Acknowledgements}
	The authors would like to thank Jefferson Baudin, Fabio Bernasconi, and
	J\'er\'emy Blanc for interesting discussions.
	The first author is supported by DFG, project number 509752046 and the
	second author was partially supported by the SNSF grant 200020\_192217 ``Geometrically ruled surfaces''.

\section{Preliminaries}

In the following, we introduce the basic notions, we use allover the article.
All varieties, morphisms and rational maps are defined over an algebraically closed field $\kk$. Here, 
a \emph{variety}
is a (not necessarily irreducible) reduced, separated scheme of finite type over $\kk$.
Throughout the whole paper $X$ denotes an irreducible variety 
and $\Bir(X) = \Bir_{\kk}(X)$ 
denotes the group of birational transformations of $X$.

\medskip

	Let $X'$ be another irreducible variety, and let $S \subseteq \Bir(X)$, $T \subseteq \Bir(X')$
	be closed subsets. We say that a map $\eta \colon S \to T$ \emph{preserves algebraic families},
	if for every morphism $\rho \colon V \to \Bir(X)$ with image in $S$, the composition
	$\eta \circ \rho$ is a morphism to $\Bir(Y)$. Note, if $\eta \colon S \to T$ preserves
	algebraic families, then it is continuous. For example, conjugation with a birational
	map $X \dashrightarrow X'$ yields a group isomorphism $\Bir(X) \to \Bir(X')$ that preserves
	algebraic families.

	We endow $\Bir(X) \times \Bir(X)$ with the induced topology of $\Bir(X \times X)$
	under the injective group homomorphism
	$(\varphi_1, \varphi_2) \mapsto ((x_1, x_2) \dashmapsto 
	(\varphi_1(x_1), \varphi_2(x_2)))$ and get the following basic properties:

	\begin{proposition} $ $
		\label{Prop.Prel}
		\begin{enumerate}
		\item \label{Prop.Prel_product_closed} The subgroup $\Bir(X) \times \Bir(X)$ of $\Bir(X \times X)$ is closed and the induced topology on $\Bir(X) \times \Bir(X)$ is the finest topology
		such that all maps $\rho_1 \times \rho_2 \colon V \to \Bir(X) \times \Bir(X)$
		are continuous for all morphisms $\rho_1, \rho_2 \colon V \to \Bir(X)$.
		Moreover, the product maps $\rho_1 \times \rho_2$ are exactly the morphisms to
		$\Bir(X \times X)$ with image in $\Bir(X) \times \Bir(X)$.
		\item \label{Prop.Prel_mult_inv_cont} The composition and the inversion 
		of birational transformations 
		\[
			\begin{array}{rcl}
				\Bir(X) \times \Bir(X) & \to & \Bir(X) \\
				(\varphi, \psi) & \mapsto & \varphi \circ \psi
			\end{array} \, , 
			\quad \quad \quad 
			\begin{array}{rcl}
				\Bir(X) & \to & \Bir(X) \\
				\varphi & \mapsto & \varphi^{-1}
			\end{array}
		\]
		preserve algebraic families.
		\item \label{Prop.Prel_diagonal} The diagonal $\Delta \subseteq \Bir(X) \times \Bir(X)$ 
		is closed and points in $\Bir(X)$ are closed.
		\item \label{Prop.Prel_Closure_product} For subsets $S_1, S_2 \subseteq \Bir(X)$ the 
			closure of $S_1 \times S_2$
			 in $\Bir(X) \times \Bir(X)$ is equal to the product 
			 $\overline{S_1} \times \overline{S_2}$.
		\item \label{Prop.Prel_connected_prod} For connected (irreducible) 
			subsets $S_1, S_2 \subseteq \Bir(X)$
			the product $S_1 \times S_2$ is a connected (irreducible) 
			subset of $\Bir(X) \times \Bir(X)$.
	\end{enumerate}
\end{proposition}

\begin{proof} 
	\eqref{Prop.Prel_product_closed}-\eqref{Prop.Prel_diagonal}: 
	This can be found in the proofs of~\cite[Propositions~4, 6, Remark~5, and Lemma~7]{PaRi2013Some-remarks-about}.

	\eqref{Prop.Prel_Closure_product}: This follows from the fact that 
	$\epsilon_{\varphi}, \eta_{\varphi} \colon \Bir(X) \to \Bir(X) \times \Bir(X)$ given by 
	$\epsilon_{\varphi}(\psi) = (\psi, \varphi)$,  $\eta_{\varphi}(\psi) = (\varphi, \psi)$
	are both closed topological embeddings
	for all $\varphi \in \Bir(X)$. Indeed, $\overline{S_1} \times \overline{S_2}$ is closed in
	$\Bir(X) \times \Bir(X)$, and if $Z$ is closed in $\Bir(X) \times \Bir(X)$ and contains
	$S_1 \times S_2$, then 
	$\overline{S_1} \subseteq \varepsilon_{\varphi}^{-1}(Z)$
	for all $\varphi \in S_2$ and hence $\overline{S_1} \times S_2  \subseteq Z$.
	Therefore, $\overline{S_2} \subseteq \eta_{\varphi}^{-1}(Z)$ for all
	$\varphi \in \overline{S_1}$, whence
	$\overline{S_1} \times \overline{S_2} \subseteq Z$.

	\eqref{Prop.Prel_connected_prod}: 
	For the irreducibility, see~\cite[Corollary~10]{PaRi2013Some-remarks-about}.
	For the connectedness, this follows again from the fact that 
	$\epsilon_{\varphi}, \eta_{\varphi} \colon \Bir(X) \to \Bir(X) \times \Bir(X)$ 
	are closed topological embeddings. Indeed:
	assume $A, B \subseteq S_1 \times S_2$ are 
	closed disjoint subsets such their union is equal to $S_1 \times S_2$.
	Then for every $\varphi \in S_1$ we have that either $\{ \varphi \} \times S_2$
	lies in $A$ or in $B$. Since this is also the case for $S_1 \times \{ \varphi \}$,
	$\varphi \in S_2$, we get that either $A$ or $B$ is equal to $S_1 \times S_2$.
\end{proof}

An immediate consequence of 
Proposition~\ref{Prop.Prel}\eqref{Prop.Prel_mult_inv_cont}\eqref{Prop.Prel_Closure_product} 
is that the closure of any subgroup of $\Bir(X)$ is again a subgroup.

\begin{remark}
	\label{Rem.InclusioAut_BIr}
	Let $X$ be an irreducible variety. On $\Aut(X)$ we have a similar topology to the Zariski-topology
	on $\Bir(X)$ that we call Zariski-topology as well. Note that the natural inclusion $\Aut(X) \to \Bir(X)$ is continuous. However, it is unclear, if a morphism $A \to \Bir(X)$ with image in 
	$\Aut(X)$ comes from an $A$-isomorphism $A \times X \to A \times X$. 
\end{remark}

We finish this section with the following notion of dimension due to 
Ramanujamn~\cite{Ra1964A-note-on-automorp} 
adapted for the groups of birational transformations:

\begin{definition}
	Let $S \subseteq \Bir(X)$ be any subset. We define its \emph{dimension} by
	\[
		\dim S \coloneqq \sup \Bigset{d \in \NN_0}{
		\begin{array}{l}
		\textrm{there is an injective morphism $\rho \colon V \to \Bir(X)$} \\
			\textrm{such that $d = \dim V$ and $\rho(V) \subseteq S$}
		\end{array}
		}
	\] 
	The supremum does not change if we assume in addition that $V$ is irreducible.
\end{definition}

\section{\texorpdfstring{Morphisms to $\Bir(X)$ and the Hilbert scheme}
{Morphisms to Bir(X) and the Hilbert scheme}}
\label{Sec.Morphisms_Hilbert}

In the following section we study $\Bir(X)$ using the Hilbert scheme of $X \times X$
as in~\cite{Ha1987On-the-birational-}. This will enable us in Section~\ref{Sec.Porperties_of_morphisms}
to find for every closed irreducible algebraic subset of $\Bir(X)$ 
an open dense subset that admits a nice parametrization, 
see~Proposition~\ref{Prop.Nice_Parametrization_closed_alg_subset}. More precisely, 
we construct in this section countably many morphisms with pair-wise disjoint images that cover 
$\Bir(X)$ and satisfy a certain kind of universal property, 
see the Corollaries~\ref{Cor.Hilb_p-morph},~\ref{Cor.Universal_property_Hilb}. Using this we establish a decomposition of every morphism to $\Bir(X)$ in 
Corollary~\ref{Cor.Decomp_of_morphism_weak}.

\medskip

We start with a lemma that gives a way to construct algebraic families.
For a rational map $\varphi \colon Y \dashrightarrow Z$ we denote 
by $\lociso(\varphi) \subseteq Y$ the open (possibly empty) 
set of those points in $Y$, where $\varphi$ induces a local isomorphism.

\begin{lemma}
	\label{Lem.Constr_rational_action}
	For $i=1,2$, let $\pi_i \colon Z_i \to V$ be a flat surjective morphism of irreducible
	varieties with irreducible reduced fibres.
	Assume that $\theta \colon Z_1 \dashrightarrow Z_2$ is a rational map such that 
	its domain surjects onto $V$ and $\pi_1 = \pi_2 \circ \theta$. If the restriction 
	$\theta_v \colon Z_{1,v} \dashrightarrow Z_{2,v}$ to the fibres is birational for all $v \in V$,
	then $\theta$ is birational and $\lociso(\theta)$ surjects onto $V$. 
\end{lemma}

\begin{proof}
	Consider the subset of the domain where $\theta$ is \'etale
	\[ 
		E \coloneq \set{ z_1 \in \dom(\theta)}{\textrm{$\theta$ is \'etale at $z_1$} } \, .
	\]
	By~\cite[Exp. I, Proposition~4.5]{GrRa2003Revetements-etales}, the set $E$ is open in $\dom(\theta)$.
	For $v \in V$ and $z_1 \in Z_{1, v}$ we have that $\theta$ is \'etale at $z_1$ if and only if 
	$\theta_v$ is \'etale at $z_1$ by~\cite[Exp.~I, Corollaire~5.9]{GrRa2003Revetements-etales}.
	In particular, $E$ surjects onto $V$.

	Now, for $v \in V$ consider the open dense subset 
	$E_v \coloneqq Z_{1, v} \cap E$ of $Z_{1, v}$. 
	Then $\theta_v |_{E_v} \colon E_v \to Z_{2, v}$ is a birational, \'etale
	morphism. Using that \'etale morphisms are locally 
	standard \'etale (see e.g.~\cite[Lemma 29.36.15]{St2024The-Stacks-project}), we conclude that 
	$\theta_v |_{E_v}$ is an open immersion. Hence, $\theta |_E \colon E \to Z_2$
	is an open embedding by~\cite[Exp.~I, Proposition~5.7]{GrRa2003Revetements-etales}.
\end{proof}

The most important algebraic families of birational transformations are para\-metrized by algebraic groups
and compatible with the group multiplication.
The following definition is due to Demazure~\cite[Definition~1, p. 514]{De1970Sous-groupes-algeb}:

\begin{definition}
A rational map $\alpha \colon G \times X \dashrightarrow X$ for an algebraic group $G$ is called a
\emph{rational $G$-action} if the rational map
\begin{equation}
	\label{Eq.Rational_action}
	\theta \colon G \times X \dashrightarrow G \times X \, , \quad
	(g, x) \dashmapsto (g, \alpha(g, x)) 
\end{equation}
is dominant and the following diagram commutes
\[
	\xymatrix@=15pt{
		G \times G \times X \ar@{-->}[d]_-{\id_G \times \alpha} \ar[rrrr]^-{(g_1, g_2, x) \mapsto (g_1 g_2, x)} 
		&&&& G \times X \ar@{-->}[d]^-{\alpha} \\
		G \times X \ar@{-->}[rrrr]^-{\alpha} &&&& X
	}
\]
\end{definition}

\begin{corollary}
	\label{Cor.Rational_action_is_family}
	The rational map~\eqref{Eq.Rational_action}
	is an algebraic family of birational transformations of $X$ parametrized by $G$
	and $\rho_{\theta} \colon G \to \Bir(X)$ is a group homomorphism.
\end{corollary}

\begin{proof}
	The domain of $\theta$ surjects onto $G$ 
	and for all $g \in G$ we have that
	the restriction $\theta_g \colon X \dashrightarrow X$ of $\theta$ over $g$ is dominant
	by \cite[Lemme~1, p. 515]{De1970Sous-groupes-algeb} and using \cite[(PO 2'), p. 514]{De1970Sous-groupes-algeb} 
	we get $\theta_g \circ \theta_{h} = \theta_{gh}$ for all $g, h \in G$.
	Hence, $\theta_g$ is birational for all $g \in G$ and thus $\theta$ is an algebraic family by 
	Lemma~\ref{Lem.Constr_rational_action}; moreover, $\rho_{\theta} \colon G \to \Bir(X)$ is a group homomorphism.
\end{proof}

Let us assume for now that $X$ is projective. In this case, we can construct countably many injective morphisms that cover 
$\Bir(X)$ by using Hilbert schemes. We may see
$\Bir(X)$ as an open subset of the Hilbert scheme of $X \times X$ 
(see \cite[Proposition~1.7]{Ha1987On-the-birational-}, which also works in positive characteristic).
For every Hilbert polynomial $p \in \QQ[X]$ denote by $\Hilb_p$ the intersection 
of $\Bir(X)$ with the connected component of the Hilbert scheme corresponding to $p$
(for this we fix once and for all a closed embedding of $X \times X$ into some
$\PP^N$). Note that the $\Hilb_p$, $p \in \QQ[T]$ give a partition of $\Bir(X)$.

\begin{corollary}
	\label{Cor.Hilb_p-morph}
	Assume that $X$ is projective.
	For every $p \in \QQ[T]$, the inclusion $\iota_p \colon \Hilb_p \to \Bir(X)$ is a morphism.	
\end{corollary}

\begin{proof}
	Let $Z_p \subseteq \Hilb_p \times X \times X$ be the intersection of 
	the universal family  over the $p$-th component of the Hilbert-scheme of $X \times X$
	with $\Hilb_p \times X \times X$. Then for $i = 1, 2$, the $i$-th projection
	\[
		q_i \colon Z_p \to \Hilb_p \times X \, , \quad (f, x_1, x_2) \mapsto (f, x_i)	
	\]
	is a morphism that restricts over every $f \in \Hilb_p$ to 
	a birational morphism $Z_{p, f} \to X$ 
	(see e.g. \cite[Proof of Proposition~1.7]{Ha1987On-the-birational-}).
	By Lemma~\ref{Lem.Constr_rational_action}, $q_i$ is a birational
	morphism such that $\lociso(q_i)$ surjects onto $\Hilb_p$. Hence,
	$q_2 \circ q_1^{-1} \colon \Hilb_p \times X \dashrightarrow \Hilb_p \times X$
	is an algebraic family of birational transformations of $X$
	and the associated morphism is equal to $\iota_p$.
\end{proof}

Moreover, the morphisms $\iota_p \colon \Hilb_p \to \Bir(X)$ in Corollary~\ref{Cor.Hilb_p-morph} 
satisfy the following kind of universal property 
(which will be proven in Corollary~\ref{Cor.Decomp_of_morphism_weak}):

\begin{definition}
	A morphism $\rho \colon V \to \Bir(X)$ is called \emph{rationally universal}, if for every
	morphism $\varepsilon \colon A \to \Bir(X)$ with irreducible $A$ and 
	image inside $\rho(V)$, there exists a unique rational map
	$f \colon A \dashrightarrow V$ such that $\varepsilon = \rho \circ f$.
\end{definition}

Note that if $\rho \colon V \to \Bir(X)$ is a rationally universal morphism and if
$V' \subseteq V$ is a locally closed subset with $V' = \rho^{-1}(\rho(V'))$,  then
the restriction $\rho |_{V'} \colon V' \to \Bir(X)$ is a rationally universal morphism as well.
The following lemma from algebraic geometry will be important for the proof that 
$\iota_p$ is rationally universal. 

\begin{lemma}
	\label{Lem.open}
	Let $f \colon E \to B$ be a dominant morphism of irreducible varieties. Assume that there is an open dense $U \subseteq E$.
	Then there exists an open dense subset $B_0 \subseteq B$ such that for all $b \in B_0$ we have that
	$f^{-1}(b)  \cap U$ is dense in $f^{-1}(b)$.
\end{lemma}

\begin{proof}[Proof of Lemma~\ref{Lem.open}]
	Let $F = E \setminus U$. As $E$ is irreducible, $F$ has dimension strictly smaller than $E$. Let 
	$F_0$ be the union of all irreducible components of $F$ that dominate $B$ and let $F_1$ be the union of 
	all other irreducible components of $F$. In case $F_0$ is empty, there exists an open dense subset $B_0 \subseteq B$ such that
	$f^{-1}(B_0) \subseteq U$ and hence we are done. Thus, we may assume that $F_0$ is non-empty.
	By generic flatness,
	there exists a dense open subset $B_0$ in $B$ such that 
	\[
		f^{-1}(B_0) \xrightarrow{f} B_0 \quad \textrm{and} \quad 
		F_0 \cap f^{-1}(B_0) \xrightarrow{f} B_0
	\]
	are flat and surjective, and $f(F_1) \cap B_0$ is empty. As the fibres of a 
	surjective flat morphism of irreducible varieties are equidimensional 
	(see \cite[Ch. III, Proposition~9.5]{Ha1977Algebraic-geometry}), it follows for all $b \in B_0$  that
	every irreducible component of $F_0 \cap f^{-1}(b)$ has dimension $< \dim E - \dim B$, whereas
	every irreducible component of $f^{-1}(b)$ (and hence of $U \cap f^{-1}(b)$) 
	has dimension $\dim E - \dim B$.
	Since $f^{-1}(b)$ is the union of $F_0 \cap f^{-1}(b)$ and $U \cap f^{-1}(b)$, the statement follows.
\end{proof}

\begin{lemma}
	\label{Lem.Hilb}
	Assume that $X$ is projective.
	Let $\rho \colon W \to \Bir(X)$ be a morphism for an irreducible $W$. Then there
	exists a rational map 
	$\lambda \colon W \dashrightarrow \Hilb_p$ for some $p \in \QQ[T]$ such that
	$\rho = \iota_p \circ \lambda$.
\end{lemma}

\begin{proof}
	Let $\theta$ be the 
	algebraic family of birational transformations of $X$ 
	with $\rho_{\theta} = \rho$. Consider the graph
	\[
		\Gamma_{\theta} = 
		\set{(w, x_1, x_2) \in \lociso(\theta) \times X}{\theta(w, x_1) = (w, x_2)} 
		\subseteq \lociso(\theta) \times X \, ,
	\]
	and let $\overline{\Gamma_{\theta}}$ be the closure inside $W \times X \times X$.
	Then there is an open dense subset $W' \subseteq W$
	such that the projection $\pi \colon \overline{\Gamma_{\theta}} \to W$ is flat over $W'$ and
	$\Gamma \cap \pi^{-1}(w')$ is dense in $\pi^{-1}(w')$ 
	for all $w' \in W'$
	(see Lemma~\ref{Lem.open}). Hence, $\pi^{-1}(W') \to W'$ is a flat family over 
	$W'$ in the sense of \cite[Definition~2.1]{Ha1987On-the-birational-} and thus
	there exists a Hilbert polynomial $p \in \QQ[T]$ and a morphism 
	$\lambda \colon W' \to \Hilb_p$ such that $\pi^{-1}(W') \to W'$ is the pull-back
	of the universal family over $\Hilb_p$ via $\lambda$ 
	(see \cite[Proposition~2.2]{Ha1987On-the-birational-}). This shows that 
	$\rho |_{W'} = \iota_p \circ \lambda$.
\end{proof}

\begin{corollary}
	\label{Cor.Universal_property_Hilb}
	Assume that $X$ is projective.
	Then $\iota_p \colon \Hilb_p \to \Bir(X)$ is rationally universal
	for all $p \in \QQ[T]$.
\end{corollary}

\begin{proof}
	We use Lemma~\ref{Lem.Hilb}, the fact that $\iota_p(\Hilb_p)$ and $\iota_q(\Hilb_q)$
	are disjoint for distinct $p, q \in \QQ[T]$, and the injectivity of $\iota_p$.
\end{proof}

\begin{corollary}
	\label{Cor.Decomp_of_morphism_weak}
	Let $\rho \colon W \to \Bir(X)$ be a morphism with irreducible $W$. Then there
	exists a dominant rational map $\lambda \colon W \dashrightarrow U$
	and an injective rationally universal 
	morphism $\eta \colon U \to \Bir(X)$ such that $\rho = \eta \circ \lambda$.
\end{corollary}

\begin{proof}
	We may and will assume that $X$ is projective.
	Now, we apply Lemma~\ref{Lem.Hilb} to $\rho \colon W \to \Bir(X)$
	in order to obtain a rational map $\lambda \colon W \dashrightarrow \Hilb_p$ with 
	$\rho = \iota_p \circ \lambda$ for some $p \in \QQ[T]$. We let now $U$ 
	be the closure of the image of $\lambda$ in $\Hilb_p$, and we let 
	$\eta$ be the restriction of $\iota_p$ to $U$.
\end{proof}

\section{\texorpdfstring{An exhaustive family of morphisms to $\Bir(X)$}
{An exhaustive family of morphisms to Bir(X)}}
\label{Sec.ExhaustiveFamily}

The next results (until Corollary~\ref{Cor.Inductive-limit-topology}) 
generalize~\cite[\S2.1-2.3]{BlFu2013Topologies-and-str} from 
$\PP^n$ to any irreducible projective variety $X$. We follow the general strategy 
from~\cite{BlFu2013Topologies-and-str}. However, some steps require some non-trivial adaptations.

More precisely, we construct a countable family of morphisms to $\Bir(X)$ such that
their images form an exhaustive chain of closed irreducible algebraic subsets
and $\Bir(X)$ carries the inductive-limit topology with respect to these images, 
see Corollary~\ref{Cor.Inductive-limit-topology}. As an application, we show among other things 
that closed subsets in $\Bir(X)$ have only countably many irreducible components, 
see Corollary~\ref{Cor.connected_components}, and that
every closed connected subgroup of $\Bir(X)$ can be exhausted by an ascending 
chain of closed irreducible algebraic subsets, see Corollary~\ref{Cor.Chain_irreducible_alg_subsets}.

\medskip

We fix once and for all a non-degenerate closed embedding $X \subseteq \PP^n$
(i.e.~$X$ is not contained in any hyperplane),
and we denote by $I(X)$ the homogeneous ideal in $\kk[x_0, \ldots, x_n]$ generated by
those homogeneous polynomials that vanish on $X$. Moreover, we consider
the  homogeneous coordinate ring associated to $X$
\[
	\kk[x_0, \ldots, x_n] / I(X) = 
	\bigoplus_{d \geq 0} \kk[x_0, \ldots, x_n]_d / I(X)_d \, ,
\]
where
\[
	I(X)_d \coloneqq I(X) \cap \kk[x_0, \ldots, x_n]_d \, ,	
\]
and $\kk[x_0, \ldots, x_n]_d$ is the vector space of homogeneous polynomials of degree $d$.

\begin{definition}
	\label{Def.W_dH_d}
	Fix an integer $d \geq 1$. Denote 
	$P_d = \PP((\kk[x_0, \ldots, x_n]_d / I(X)_d)^{n+1})$.
	Let $W_d \subseteq P_d$ be the closed subvariety of those
	$f = (f_0, \ldots, f_n) \in P_d$ 
	such that $r(f_0, \ldots, f_n) = 0$ in $\kk[x_0, \ldots, x_n] / I(X)$
	for all homogeneous $r \in I(X)$. Then, every $f = (f_0, \ldots, f_n) \in W_d$
	defines a rational map 
	\[
		\psi_f \colon X \dashrightarrow X \, , \quad [a_0: \cdots: a_n] \dashmapsto 
		[f_0(a_0, \ldots, a_n) : \cdots : f_n(a_0, \ldots, a_n)] \, .
	\]
	Conversely, every rational map $X \dashrightarrow X$ is of the above form
	for some $f \in W_d$.
	Moreover, let $H_d \subseteq W_d$ be the subset of those $(n+1)$-tuples $f = (f_0, \ldots, f_n)$
	such that $\psi_f \colon X \dashrightarrow X$ is birational, and denote by
	$\pi_d \colon H_d \to \Bir(X)$ the map $f \mapsto \psi_f$.
\end{definition}

Obviously the maps $\pi_d \colon H_d \to \Bir(X)$ depend on the choice of the embedding of 
$X$ into $\PP^n$.

\begin{lemma}
	\label{Lem.H_d}
	With the notation of Definition~\ref{Def.W_dH_d} the following holds:
	\begin{enumerate}[left=0pt]
		\item \label{Eq.H_d_locally_closed} The set $H_d$ is locally closed in $W_d$.
		\item \label{Eq.H_d_family} The assignment
				\[
					\theta \colon H_d \times X \dasharrow H_d \times X \, , \quad
					(f, x) \dashmapsto (f, \psi_f(x))
				\] 
				defines an algebraic family of birational transformations of $X$ parametrized by $H_d$. 
				In particular, 
				$\pi_d = \rho_{\theta} \colon H_d \to \Bir(X)$ is a morphism.
		\item \label{Eq.closed} If $F \subseteq H_d$ is closed, then 
		 		$\pi_m^{-1}(\pi_d(F))$ is closed in $H_m$ for all $m \geq 1$.
	\end{enumerate}
\end{lemma}

\begin{proof} $ $
	\eqref{Eq.H_d_locally_closed}: We start with the following claim:
	
	\begin{claim}
		\label{Claim.etale}
		Let $U_d \subseteq W_d$ be the set of those $f = (f_0, \ldots, f_n)$
		such that there exists $x \in X$ with the property that 
		$f_i$ does not vanish in $x$ for some $i$ and 
		$\psi_f$ is \'etale at $x$. Then $U_d$ is open in $W_d$.
	\end{claim}

	\begin{proof}
		For $x \in X$ let $W_{d, x}$ be the set of those 
		$f = (f_0, \ldots, f_n) \in W_d$ such that $f_i$ does not vanish in $x$ for some $i$.
		Hence, $W_{d, x}$ is open in $W_d$. Now, consider the rational map
		\[
			\begin{array}{rcl}
			\theta \colon W_{d, x} \times X &\dashrightarrow& W_{d, x} \times X \, , \\
			(f, [a_0: \ldots: a_n]) &\dashmapsto& 
			(f, [f_0(a_0, \ldots, a_n): \ldots: f_n(a_0, \ldots, a_n)]) \, .
			\end{array}
		\]
		By construction of $W_{d, x}$ we get that $W_{d, x} \times \{x\}$ lies in the domain of $\theta$.
		For any $f \in W_{d, x}$, the following holds: 
		$\psi_f$ is \'etale at $x$ if and only
		if $\theta$ is \'etale at $(f, x)$ by 
		\cite[Exp.~I, Corollaire~5.9]{GrRa2003Revetements-etales}.
		The set of points in the domain of $\theta$ where $\theta$ is \'etale forms
		an open subset \cite[Exp.~I, Proposition~4.5]{GrRa2003Revetements-etales}.
		Hence, we conclude that the subset $U_{d, x} \subseteq W_{d, x}$ of those $f \in W_{d, x}$ such that $\psi_f$ is \'etale at $x$ forms an open subset of $W_{d, x}$.
		Since $U_d$ is the union of all $U_{d, x}$, $x \in X$, the claim follows.
	\end{proof}

	There exists $D \geq 0$ such that for all $f \in W_d$ with birational $\psi_f$
	there exists $g \in W_{D}$ such that $\psi_g = \psi_{f^{-1}}$, 
	see~\cite[Proposition~2.2]{HaSi2017Bounds-on-degrees-}.
	For $(g, f) \in W_{D} \times W_d$, we denote
	\[
		h = h_{g, f} = (h_0, \ldots, h_n) = (g_0(f_0, \ldots, f_n), \ldots, g_n(f_0, \ldots, f_n)),
	\]
	which is a well-defined element of $(\kk[x_0, \ldots, x_n]_{Dd} / I(X)_{Dd})^{n+1}$
	up to multiplication with a non-zero scalar ($h$ is possibly $0$).
	Moreover, $r(h_0, \ldots, h_n) = 0$ in $\kk[x_0, \ldots, x_n] / I(X)$ for all homogeneous $r \in I(X)$.
	Thus, in case $h$ is non-zero, we may consider $h$ as an element of $W_{dD}$.
	
	Let $Y \subseteq W_{D} \times W_d$ be the closed
	subset of those $(g, f)$ such that 
	$h_{g, f} = (h_0, \ldots, h_n)$ satisfies 
	$h_i x_j = h_j x_i$ in the vector space 
	$\kk[x_0, \ldots, x_n]_{Dd+1} / I(X)_{Dd+1}$ for all $i, j$.
	Let $\pr_2 \colon W_{D} \times W_d \to W_d$ be the projection to the second factor.
	Since $W_D$ is projective, $\pr_2(Y)$ is closed in $W_d$.
	In order to show that $H_d$ is locally closed in $W_d$, it  is enough to show
	that $H_d$ is the intersection of $\pr_2(Y)$ and $U_d$ (where $U_d$ is defined as in 
	Claim~\ref{Claim.etale}). 

	Let $(g, f) \in W_D \times W_d$ such that $f \in U_d$.
	If $h_{g, f}$ vanishes, then $\psi_f \colon X \dasharrow X$ maps
	$\dom(\psi_f)$ into the proper closed subset $V_X(g_0, \ldots, g_n)$ of $X$, 
	and thus $\psi_f$ would not be dominant, which contradicts $f \in U_d$.
	Hence, $h_{g, f}$ does not vanish.
	If moreover  $(g, f) \in Y$, then $\psi_h$ is equal to the identity on $X$
	(where $h = h_{g, f}$). Since $\psi_h = \psi_g \circ \psi_f$ we obtain that 
	$\psi_f$ is birational. This shows that
	$f \in H_d$ and therefore $\pr_2(Y) \cap U_d \subseteq H_d$.

	On the other hand, let $f \in H_d$. As $\psi_f$ is birational, it follows that $f \in U_d$,
	and there exists $g \in W_D$ such that
	$\psi_g \circ \psi_f$ represents the identity on $X$. As in the last paragraph 
	$h = h_{g, f}$ does not vanish
	and since $\psi_g \circ \psi_f = \psi_h$, we get $(g, f) \in Y$.  Hence, we have seen
	that $H_d \subseteq \pr_2(Y) \cap U_d$.

	\eqref{Eq.H_d_family}: The domain of the
	rational map 
	\[
		P_d \times X \dashrightarrow P_d \times \PP^n \, , \ 
		(f, [a_0: \cdots: a_n]) \dashmapsto 
		(f, [f_0(a_0, \ldots, a_n) : \cdots : f_n(a_0, \ldots, a_n)])
	\]
	contains those pairs such that $f_i(a_0, \ldots, a_n)$ is non-zero for some $i =0, \ldots, n$. 
	Hence, the restriction to the locally closed subset $H_d \times X$ yields a rational
	$H_d$-map $\theta \colon H_d \times X \dasharrow H_d \times X$ whose domain surjects to 
	$H_d$ (by~\eqref{Eq.H_d_locally_closed} the set $H_d$ is locally closed in $W_d$). 
	The statement follows now from Lemma~\ref{Lem.Constr_rational_action}.

	\eqref{Eq.closed} Let $\bar{F}$ be the closure of $F$ in $W_d$. We consider the closed
	subset $Z \subseteq W_m \times \bar{F}$ that is given by the pairs $(g, f)$ such that
	$g_i f_j = g_j f_i$ in $\kk[x_0, \ldots, x_n]_{md} / I(X)_{md}$ for all $i, j$. Hence, 
	for $(g, f) \in W_m \times W_d$ we have:
	$(g, f) \in Z$
	if and only if $\psi_g$ and $\psi_f$ coincide (as rational maps $X \dasharrow X$) and 
	$f \in \bar{F}$.
	Let $\pr_1 \colon H_m \times \bar{F} \to H_m$ be the projection to the first factor.
	As $\bar{F}$ is projective, $\pr_1$ is closed. In particular, $\pr_1 (Z \cap (H_m \times \bar{F}))$
	is closed in $H_m$. Moreover, 
	$\pi_m^{-1}(\pi_d(F)) = \pr_1 (Z \cap (H_m \times \bar{F}))$. Indeed, if 
	$(g, f) \in Z \cap (H_m \times \bar{F})$, then $\psi_f = \psi_g$ is birational and hence 
	$f \in \bar{F} \cap H_d = F$.
\end{proof}

\begin{lemma}
	\label{Lem.Family_factorizing_through_H_d}
	Let $A$ be a variety and let 
	$\theta \colon A \times X \dasharrow A \times X$ be
	an algebraic family of birational transformations of $X$. 
	Then $A$ admits an open affine covering $(A_i)_{i \in I}$
	such that for all $i \in I$, there exist $d_i \geq 1$ and a morphism 
	$\rho_i \colon A_i \to H_{d_i}$ such that $\rho_{\theta} |_{A_i} = \pi_{d_i} \circ \rho_i$. 
\end{lemma}

\begin{proof}
	We fix $a_0 \in A$. By definition, there exists $p_0 \in X$ such that 
	$(a_0, p_0) \in \lociso(\theta)$.
	Choose an open affine neighborhood $A_0$ of $a_0$ in $A$. We choose coordinates
	of $\PP^n$ in such a way that $p_0 = [1:0 \ldots:0] \in \PP^n$. Let 
	$\iota \colon \AA^n \to \PP^n$ be
	the open embedding that is given by $(y_1, \ldots, y_n) \mapsto [1: y_1: \ldots: y_n]$.
	Then, there exists a rational map 
	$\lambda \colon A_0 \times \iota^{-1}(X) \dashrightarrow \iota^{-1}(X)$ that is defined
	at $(a_0, 0, \ldots, 0)$ such that
	$\pr_2 \circ \theta \circ (\id_{A_0} \times \iota |_{\iota^{-1}(X)})$ 
	is equal to $\iota |_{\iota^{-1}(X)} \circ \lambda$,
	where $\pr_2 \colon A_0 \times X \to X$ denotes the projection onto the second factor.
	Hence, there exist polynomials $f_0, \ldots, f_n \in \kk[A_0][y_1, \ldots, y_n]$ 
	such that $f_0(a_0, 0, \ldots, 0) \neq 0$ and $\lambda$ is given by
	\[
		(a, (y_1, \ldots, y_n))	\dashmapsto \left( 
			\frac{f_1(a, y_1, \ldots, y_n)}{f_0(a, y_1, \ldots, y_n)}, \ldots, 
			\frac{f_n(a, y_1, \ldots, y_n)}{f_0(a, y_1, \ldots, y_n)}
		\right)
	\]
	in a neighborhood of $(a_0, 0, \ldots, 0)$ in $A_0 \times \iota^{-1}(X)$.
	After homogenizing the $f_0, \ldots, f_n$ we may assume that there
	exist $d \geq 0$ and homogeneous polynomials $h_0, \ldots, h_n \in \kk[A_0][x_0, \ldots, x_n]$
	of degree $d$ such that $\theta$ is given by
	\[
		(a, [x_0: \ldots: x_n]) \mapsto 	
		(a, [h_0(a, x_0, \ldots, x_n): \ldots : h_n(a, x_0, \ldots, x_n)])
	\]
	in a neighborhood of $(a_0, p_0)$ in $A_0 \times X$.
	Moreover, after possibly shrinking $A_0$ we may assume that
	for every $a \in A_0$ there exists $i$ with $h_i(a, p_0) \neq 0$.
	The homogeneous polynomials $h_0, \ldots, h_n$ give rise to a morphism
	$\rho_0 \colon A_0 \to H_d$ such that $\pi_d \circ \rho_0 = \rho_{\theta} |_{A_0}$.
	Since $a_0 \in A$ was arbitrary, the statement follows.
\end{proof}

As in \cite[Corollary~2.7-2.9, Proposition~2.10]{BlFu2013Topologies-and-str}, we deduce from Lemma~\ref{Lem.H_d}
and Lemma~\ref{Lem.Family_factorizing_through_H_d} the following consequence.
In case $\kk[x_0, \ldots, x_n]/ I(X)$ is factorial, the result can be found 
in~\cite[Corollary~3.17, see also Defintion~3.3]{HaMo2024Open-loci-of-ideal}.

\begin{corollary} $ $
	\label{Cor.Inductive-limit-topology}
	\begin{enumerate}[left=0pt]
		\item A subset $A \subseteq \Bir(X)$ is closed in $\Bir(X)$ if and only if
			  $\pi_d^{-1}(A)$ is closed in $H_d$ for all $d \geq 1$.
		\item For every $d \geq 1$ the morphism 
			  $\pi_d \colon H_d \to \Bir(X)$ is closed
	\end{enumerate}
	In particular, $\Bir(X)$ carries the inductive-limit topology with respect
	to the filtration by the closed subsets $\pi_1(H_1) \subseteq \pi_2(H_2) \subseteq \ldots$
	in $\Bir(X)$. \qed
\end{corollary}

\begin{remark}	
\label{Rem.Similar_as_ind-group}
Note, if $X$ is an affine variety, then $\Aut(X)$ has a natural structure of a so-called 
ind-group,~i.e. $\Aut(X)$ can be filtered by a countable union of affine varieties 
$V_1 \subseteq V_2 \subseteq \ldots$, where $V_d$ is closed in $V_{d+1}$ for all $d$, such that 
multiplication and inversion maps are compatible with these filtrations. 
Moreover, the 
morphisms $A \to \Aut(X)$
correspond to the morphisms of varieties $A \to V_d$, where $d \geq 1$, 
see~\cite[Theorem~5.1.1.]{FuKr2018On-the-geometry-of}. Although, $\Bir(X)$ cannot have the structure
of an ind-group described above (see~\cite[Proposition~3.4]{BlFu2013Topologies-and-str}), Corollary~\ref{Cor.Inductive-limit-topology} and Lemma~\ref{Lem.Family_factorizing_through_H_d} say roughly speaking that 
$\Bir(X)$ has still a very similar structure. In fact: $\Bir(X)$ carries the inductive
limit topology of the images $V_1 \subseteq V_2 \subseteq \cdots$ of the closed morphisms 
$\pi_d \colon H_d \to \Bir(X)$
and the morphisms $A \to \Bir(X)$ correspond 
to families of morphisms $(f_i \colon A_i \to H_{d_i})_{i \in I}$ 
such that $(A_i)_{i \in I}$ is an open affine cover of $A$ and the maps
$\pi_{d_j} \circ f_j$, $\pi_{d_i} \circ f_i$ coincide on $A_i \cap A_j$ for all $i, j \in I$.
\end{remark}

Another immediate consequence of Corollary~\ref{Cor.Inductive-limit-topology} is:

\begin{corollary}
	\label{Cor.Closure_alg_subset}
	The closure of an algebraic subset of $\Bir(X)$ is algebraic.
\end{corollary}

\begin{proof}
	The image of a morphism is contained in some $\pi_d(H_d)$ by 
	Lemma~\ref{Lem.Family_factorizing_through_H_d}.
\end{proof}

In the next lemma we describe the fibres of $\pi_d \colon H_d \to \Bir(X)$.

\begin{lemma}
	\label{Lem.fibres_of_pi_d}
	Every fibre of $\pi_d \colon H_d \to \Bir(X)$ is either empty, or isomorphic to a 
	projective space.
\end{lemma}

\begin{proof}
	Let $f = (f_0, \ldots, f_n)$ be an element of $H_d$. We may assume that $f_0$ is non-zero.
	Note that
	\[
		\Gamma \coloneqq \Bigset{g \in \PP(((\kk[x_0, \ldots, x_n] / I(X))_d)^{n+1})}{
		\begin{array}{l}
			\textrm{$g_i f_j - g_j f_i = 0$} \\
			\textrm{in $\kk[x_0, \ldots, x_n]_{d^2} / I(X)_{d^2}$} \\
			\textrm{for all $i, j$}
		\end{array}
		}	
	\]
	is a projective linear subspace and hence isomorphic to a projective space.

	Let $g \in \Gamma$. Since $f_0$ is non-zero, since not all $g_0, \ldots g_n$
	are zero, and since $g_i f_0 - g_0 f_i = 0$
	in the integral domain $\kk[x_0, \ldots, x_n] / I(X)$, we observe that $g_0$ is non-zero as well.
	For a homogeneous $r \in I(X)$ of degree $e$ we get that 
	\[
		f_0^e r(g_0, \ldots, g_n) = g_0^e f_0^e r\left(1, \frac{g_1}{g_0}, \ldots, \frac{g_n}{g_0}\right)
		= g_0^e f_0^e r\left(1, \frac{f_1}{f_0}, \ldots, \frac{f_n}{f_0}\right) = 
		g_0^e r(f_0, \ldots, f_n)
	\]
	vanishes inside $\kk[x_0, \ldots, x_n]/I(X)$, since $f \in H_d$. However, since $f_0$ is non-zero, we get that
	$r(g_0, \ldots, g_n)$ vanishes and hence $g \in W_d$.
	Since $g_i f_j - g_j f_i = 0$ for all $i, j$, it follows that $\psi_g = \psi_f$ is birational
	and thus $g \in H_d$. This shows that $\pi_d^{-1}(\psi_f) = \Gamma$ and hence this fibre
	is isomorphic to a projective space.
\end{proof}

\begin{remark}
	In case the homogeneous coordinate ring 
	$R \coloneqq \kk[x_0, \ldots, x_n] / I(X)$ of $X$
	is a unique factorization domain
	and $(f_0, \ldots, f_n) \in H_d$ is an element such that $\gcd(f_0, \ldots, f_n) = 1$,
	then $\pi_d^{-1}(\pi_d(f_0, \ldots, f_n))$ consists of one element. On the other hand, this does not need to be true in case $R$ is not a unique factorization domain, for example, take
	$X = V(xy-zw) \subseteq \PP^3$. Then the two distinct elements 
	\[
		(z(w+z), y^2, yz, y(w+z)) \quad \textrm{and} \quad
		(x(w+z), yw, zw, (w+z)w)
	\]
	of $H_2$ both induce the conjugate $\psi$ of the automorphism $[y, z, w] \mapsto [y, z, w + z]$
	of $\PP^2$ by the projection $X \dasharrow \PP^2$, $[x, y, z, w] \mapsto [y, z, w]$.
	However, $\gcd(w(w+z), y^2, yw, y(w+z)) = 1$ and $\gcd(x(w+z), yz, wz, z(w+z)) = 1$, since
	$\psi$ is not an automorphism of $X$.
\end{remark}

\begin{corollary}
	\label{Cor.Preimage_conn}
	Assume that $F \subseteq \pi_d(H_d)$ is a closed connected subset. Then 
	$\pi_d^{-1}(F)$ is a closed connected subset of $H_d$.
\end{corollary}

\begin{proof}
	Assume towards a contradiction that there exist disjoint non-empty closed subsets
	$A_0, A_1$ in $\pi_d^{-1}(F)$ such that their union is equal to $\pi_d^{-1}(F)$.
	As $\pi_d \colon H_d \to \Bir(X)$ is closed (see Corollary~\ref{Cor.Inductive-limit-topology}),
	it follows that $\pi_d(A_0)$, $\pi_d(A_1)$ are closed non-empty subsets of 
	$F$ and their union is equal to $F$. Since $F$ is connected,
	these sets must intersect, say in $\varphi \in F$. However, since
	$\pi_d^{-1}(\varphi)$ is connected 
	(it is even irreducible by Lemma~\ref{Lem.fibres_of_pi_d}) and since it is covered
	by the non-empty disjoint closed subsets $\pi_d^{-1}(\varphi) \cap A_0$,
	$\pi_d^{-1}(\varphi) \cap A_1$, we arrive at a contradiction.
\end{proof}

\begin{corollary}
	\label{Cor.connected_components}
	Let $Z \subseteq \Bir(X)$ be a closed subset. Then:
	\begin{enumerate}[left=0pt]
		\item \label{Cor.connected_components1} 
			$Z$ has only countably many connected components
		\item \label{Cor.connected_components1.5}
			The connected components of $Z$ are open in $Z$.
		\item \label{Cor.connected_components2}
			$Z$ is connected if and only if for each $z_0, z_1 \in Z$, there 
			exists a connected closed algebraic subset $V \subseteq \Bir(X)$
			such that $z_0, z_1 \in V \subseteq Z$.
	\end{enumerate}
\end{corollary}

\begin{proof}
	For $d \geq 1$, let $Z_d \subseteq \Bir(X)$ be the intersection of $Z$ with $\pi_d(H_d)$.
	Since $\pi_d^{-1}(Z_d)$ is closed in $H_d$, it follows that it consists only
	of finitely many connected components and thus the same holds for $Z_d$.

	\eqref{Cor.connected_components1}:
	As each connected component of $Z_d$ has to be contained in a connected component
	of $Z$ and since the $Z_d$, $d \geq 1$ exhaust $Z$, it follows that $Z$
	has at most countably many connected components. 

	\eqref{Cor.connected_components1.5} Let $\rho \colon A \to \Bir(X)$ be a morphism
	with image in $Z$. Every connected component of $A$ is mapped into a 
	connected component of $Z$. Hence, the preimage of every union of connected components
	of $Z$ under $\rho$ is the union of some connected components of $A$
	and therefore closed in $A$.

	\eqref{Cor.connected_components2}:
	Assume that $Z$ is connected, let $z_0, z_1 \in Z$ and let $Z_d = Z \cap \pi_d(H_d)$.
	
	\begin{claim}
		There exists $d \geq 1$ such that $z_0, z_1$
		are both contained in a connected component of $Z_d$.
	\end{claim}
	
	\begin{proof}
	Otherwise, for every $d \geq 1$
	we get closed disjoints subsets $A_{d, 0}$, $A_{d, 1}$ of $Z_d$ such that
	their union is equal to $Z_d$ and
	$z_i \in A_{d, i}$ for $i=0, 1$. Hence, $A_{d, i} \subseteq A_{d+1, i}$ for all $i=0, 1$
	and all $d \geq 1$. The sets
	$A_i = \bigcup_{d \geq 1} A_{d, i}$, $i = 0, 1$
	are disjoint, non-empty, and their union is equal to $Z$.
	Moreover, $A_i$, is closed in $Z$, since $A_i \cap Z_d = A_{d, i}$ for $i =0, 1$,
	see Corollary~\ref{Cor.Inductive-limit-topology}.
	This contradicts the connectedness of $Z$.
	\end{proof}

	This shows one implication, the other implication is clear.
\end{proof}

For closed subgroups in $\Bir(X)$ we can strengthen Corollary~\ref{Cor.connected_components}:

\begin{corollary}
	\label{Cor.Closed_subgroup}
	Let $G \subseteq \Bir(X)$ be a closed subgroup and let $G^\circ$ be the
	connected component of the identity. Then
	\begin{enumerate}[left=0pt]
		\item \label{Cor.Closed_subgroup1} $G^\circ$ has countable index in $G$, 
		and $G^\circ$ is open and closed in $G$.
		\item \label{Cor.Closed_subgroup2} 
		For all $g_0, g_1 \in G^\circ$ there exists an irreducible
		closed algebraic subset $W \subseteq \Bir(X)$
		such that $g_0, g_1 \in W \subseteq G^\circ$.
		\item \label{Cor.Closed_subgroup3} 
		Two irreducible algebraic subsets of $G^\circ$ are always contained in an
		irreducible algebraic subset of $G^\circ$.
	\end{enumerate}
\end{corollary}

\begin{proof}
	\eqref{Cor.Closed_subgroup1}: This follows from 
	Corollary~\ref{Cor.connected_components}\eqref{Cor.connected_components1},\eqref{Cor.connected_components1.5}.

	\eqref{Cor.Closed_subgroup2}: Let $g_0, g_1 \in G^\circ$. By Corollary~\ref{Cor.connected_components}\eqref{Cor.connected_components2} there exists
	a connected closed algebraic subset $V \subseteq \Bir(X)$
	such that $g_0, g_1 \in V \subseteq G^\circ$. Let $V_0, \ldots, V_k$
	be irreducible components of $V$ such that $g_0 \in V_0$, $g_1 \in V_k$ and
	$V_i \cap V_{i+1}$ is non-empty for all $i=0, \ldots, k-1$. By induction on $k$, 
	it is enough to consider the case $k = 1$. Let $\varphi \in V_0 \cap V_1$.
	Let $\rho_i \colon A_i \to \Bir(X)$ be a morphism with image equal to $V_i$ and let
	$a_i, b_i \in A_i$ with $\rho_i(a_i) = g_i$ and $\rho_i(b_i) = \varphi$ for $i = 0, 1$.
	Then 
	\[
		\rho \colon V_0 \times V_1 \to \Bir(X) \, , \quad 
		(v, v') \mapsto \rho_0(v) \circ \varphi^{-1} \circ \rho_1(v')
	\]
	is a morphism with $\rho(a_0, b_1) = g_0$ and $\rho(b_0, a_1) = g_1$.
	Thus, $W \coloneqq \overline{\rho(V_0 \times V_1)}$ is our desired irreducible closed
	algebraic subset of $\Bir(X)$ (see Corollary~\ref{Cor.Closure_alg_subset}).

	\eqref{Cor.Closed_subgroup3}:
	Let $A, B \subseteq G^\circ$ be irreducible algebraic subsets and let $a \in A$, $b \in B$.
	By~\eqref{Cor.Closed_subgroup2} there exist
	irreducible algebraic subsets $S, T \subseteq G^\circ$ with $\id_X, a^{-1} \in S$
	and $\id_X, b^{-1} \in T$. Then $A \circ S$ contains $\id_X$ and $A$,
	and $T \circ B$ contains $\id_X$ and $B$. Hence, $A \circ S \circ T \circ B$ is our desired
 	irreducible algebraic subset of $G^\circ$.
\end{proof}

\begin{corollary}
	\label{Cor.Chain_irreducible_alg_subsets}
	Let $G \subseteq \Bir(X)$ be a closed subgroup. Then there exists
	an ascending exhausting chain of closed algebraic subsets 
	$G_1 \subseteq G_2 \subseteq \cdots$ in $G$ such that the irreducible
	components of $G_i$ are pairwise disjoint and homeomorphic. 
	If moreover $G$ is connected, then the $G_i$
	can be chosen to be irreducible.
\end{corollary}

\begin{proof}
	First we treat the case, when $G$ is connected.
	Since $G$ is covered by countably many irreducible algebraic subsets 
	(see Corollary~\ref{Cor.Inductive-limit-topology}), 
	Corollary~\ref{Cor.Closed_subgroup}\eqref{Cor.Closed_subgroup3} above implies that
	there exists an ascending exhausting chain of irreducible algebraic subsets in $G$.
	Taking the closures of these subsets implies the second statement
	(here we use Corollary~\ref{Cor.Closure_alg_subset}).

	Now, $G$ is not necessarily connected anymore. Let 
	$G_1' \subseteq G_2' \subseteq \cdots$ be an ascending exhausting chain of closed 
	irreducible algebraic subsets of $G^\circ$. As $G^\circ$ has countable 
	index in $G$ (see Corollary~\ref{Cor.Closed_subgroup}\eqref{Cor.Closed_subgroup1}) 
	there exist countably many $s_1, s_2, \ldots \in G$ such that 
	$G$ is the disjoint union of the $s_i G^\circ$, $i \geq 1$. Now, 
	$G_1 \subseteq G_2 \subseteq \cdots$ is our desired ascending exhausting chain for $G$,
	where $G_d \coloneqq \bigcup_{i=1}^d s_i G_d'$ for all $d \geq 1$.
\end{proof}

Using the same idea from the proof of Corollary~\ref{Cor.Closed_subgroup}, 
we can demonstrate a structural result for certain subgroups of $\Bir(X)$ 
that are not necessarily closed. To do this, 
recall that an \emph{ind-variety} is the inductive limit $\varinjlim V_d$ 
of a countable sequence of varieties
$V_1 \subseteq V_2 \subseteq \cdots$ such that $V_d$ is closed in $V_{d+1}$
for all $d \geq 1$.

\begin{corollary}
	Let $G \subseteq \Bir(X)$ be a subgroup that is
	generated by irreducible algebraic subsets $S_i \subseteq \Bir(X)$
	with $\id_X \in S_i$ for $i \in \NN$. Then there exists
	and ind-variety $V = \varinjlim V_d$, where each $V_d$ is irreducible
	and a map $\rho \colon V \to \Bir(X)$ with image equal to $G$ such that
	$\rho|_{V_d} \colon V_d \to \Bir(X)$ is a morphism for all $d \geq 1$.
\end{corollary}

\begin{proof}
	We may assume that $S_i = S_i^{-1}$ by replacing $S_i$ with 
	$S_i \circ S_i^{-1}$. 
	For $i=1, 2$ let
	$\rho_i \colon V_i \to \Bir(X)$ be a morphism from an irreducible variety $V_i$ to $\Bir(X)$
	such that there
	exists $e_i \in V_i$ with $\rho_i(e_i) = \id_X$. Then 
	\[
		\rho \colon V_1 \times V_2 \to \Bir(X) \, , \quad
		(v_1, v_2) \mapsto \rho_1(v_1) \circ \rho(v_2)
	\]
	is a morphism with image $\rho_1(V_1) \circ \rho_2(V_2)$ and 
	the closed embedding $\iota \colon V_1 \to V_1 \times V_2$ given by 
	$v_1 \mapsto (v_1, e_2)$ satisfies 
	$\rho_1 = \rho \circ \iota$. Note that $\id_X$ is contained in the
	image of $\rho$. By a successive use of this construction
	we get our desired map from an ind-variety to $\Bir(X)$.    
\end{proof}

\section{\texorpdfstring{Properties of morphisms to $\Bir(X)$}{Properties of morphisms to Bir(X)}}
\label{Sec.Porperties_of_morphisms}

The main result of this section is a nice parametrization of an open dense subset of every
irreducible closed algebraic subset of $\Bir(X)$, see 
Proposition~\ref{Prop.Nice_Parametrization_closed_alg_subset}. As an application we show among other
things that morphisms map locally closed subsets to locally closed subsets, 
see~Corollary~\ref{Cor.Constructible_images}, we provide a fibre dimension formula, 
see~Corollary~\ref{Cor.Fibre-dimension-formula}, and we prove that the dimension
of a closed algebraic subset of $\Bir(X)$ is equal to its Krull dimension, see~Corollary~\ref{Cor.Krull_dimensionII}.

\medskip

We start with a result which says that two members of 
an algebraic family of birational transformations of $X$ induce the same birational transformation if
they coincide on a certain finite subset of $X$. This enables us to study morphisms
to $\Bir(X)$ via rational maps of varieties. In order to formulate it, we introduce the following
notation: If $\theta$ is an algebraic family of birational transformations of $X$ parametrized by
a variety $V$, then 
$\theta_n \colon V \times X^n \dasharrow V \times X^n$ denotes the diagonal family on 
$X^n$ induced by $\theta$, 
i.e.~$\rho_{\theta_n} \colon V \to \Bir(X^n)$ is given by 
$v \mapsto (\rho_{\theta}(v), \ldots, \rho_{\theta}(v))$ and
\[
	\lociso(\theta_n) = \set{(v, x_1, \ldots, x_n) \in V \times X^n}{(v, x_i) 
	\in \lociso(\theta) \ \textrm{for all $i=1, \ldots, n$}} \, .	
\]

\begin{lemmaAndDefinition}
	\label{Lem.FinitelyManyPoints}
	Let $\theta$ be an algebraic family of birational transformations of $X$ parametrized by a 
	variety $V$. Then there exist $n \geq 1$  and $y \in X^n$ such that 
	$(V \times \{y\}) \cap \lociso(\theta_n) \neq \varnothing$ and the following holds:
	\begin{equation}
		\label{Eq.RamanujamPoint}
		(v, y), (v', y) \in \lociso(\theta_n) \ \textrm{and} \
		\rho_{\theta_n}(v)(y) =  \rho_{\theta_n}(v')(y) \implies 
		\rho_{\theta}(v) = \rho_{\theta}(v') \, .
	\end{equation}
	We call such a point $y \in X^n$ a \emph{Ramanujam point} for $\theta$.
\end{lemmaAndDefinition}

In case the algebraic family $\theta$ is an isomorphism, this can be found in 
\cite[Lemma~1]{Ra1964A-note-on-automorp}. The Ramanujamn point $y = (y_1, \ldots, y_n) \in X^n$ 
induces a rational orbit map $\lambda \colon V \dashrightarrow X^n$, and we have the following
commutative diagram
\[
	\xymatrix@R=15pt{
		V_0  \ar[rrrrd]^-{\lambda} \ar[d]_-{\rho_\theta} \\
		\rho_{\theta}(V_0) \ar[rrrr]_-{\varphi \mapsto (\varphi(y_1), \ldots, \varphi(y_n))} &&&&  X^n 
	}
\]
for an open dense subset $V_0 \subseteq V$, where the horizontal map is injective. 

\begin{proof}[Proof of Lemma~\ref{Lem.FinitelyManyPoints}]
	Let $S \subseteq V \times V$ be the set of those $(v, v')$ such that
	$\rho_{\theta}(v) = \rho_{\theta}(v')$. Using that the diagonal is closed
	in $\Bir(X)^2$ (see Proposition~\ref{Prop.Prel}\eqref{Prop.Prel_diagonal}) we obtain that $S$
	is closed in $V^2$. For $x \in X$, let us consider:
	\[
		S_x \coloneqq \set{(v, v') \in V^2}
		{(v, x), (v', x) \in \lociso(\theta) \ \implies \ 
		\rho_{\theta}(v)(x) =  \rho_{\theta}(v')(x) } \, .
	\]
	Then the set $S_x$ is closed in $V^2$, as 
	it consist of the complement in $V^2$ of the open subset 
	$U \coloneqq (\lociso(\theta) \cap (V \times \{x\}))^2$ and a closed subset in $U$.
	Moreover, note that
	\[
		\bigcap_{x \in X_0} S_x = S
	\]
	for every dense subset $X_0 \subseteq X$. Hence, 
	we may find points $y_1, \ldots, y_n$ in the projection of $\lociso(\theta)$ to $X$ 
	such that $S_{y_1} \cap \ldots \cap S_{y_n} = S$
	and thus $y = (y_1, \ldots, y_n) \in X^n$ is a Ramanujam point for $\theta$.
\end{proof}

\begin{proposition}
	\label{Prop.Nice_Parametrization_closed_alg_subset}
	Let $Z \subseteq \Bir(X)$ be a closed, irreducible, algebraic subset.
	Then there exists an irreducible $W$, a 
	closed morphism $\rho \colon W \to \Bir(X)$ with $\rho(W) = Z$ and 
	an open dense subset $W' \subseteq W$ such that $\rho |_{W'}$ decomposes as
	\[
		\rho^{-1}(\rho(W')) = W' \xrightarrow{\lambda} U 
		\xrightarrow{\eta} \rho(W') \stackrel[\substack{\textrm{open} \\ \textrm{dense}}]{}{\subseteq} Z \, ,
	\]	
	where $\lambda$ is a finite, flat, surjective morphism (of varieties) and 
	$\eta \colon U \to \Bir(X)$ is a rationally universal morphism that induces a homeomorphism 
	$U \to \rho(W')$.
\end{proposition}

For the proof and also for future use we state the following consequence of Noether's 
normalization theorem:

\begin{lemma}[{\cite[Theorem~3.4.1]{Kr2016Algebraic-transfor}}]
	\label{Lem.Kraft}
	Let $X$, $Y$ be affine irreducible varieties and let $f \colon X \to Y$ be a dominant
	morphism. Then there exists $h \in \kk[Y]$ and a finite morphism
	$\rho \colon X_h \to Y_h \times \AA^d$ such that the following diagram
	commutes
	\[
	\xymatrix@R=15pt{
		X_h \ar[rd]_-{f} \ar[r]^-{\rho} & Y_h \times \AA^d \ar[d]^-{(y, v) \mapsto y} \\
		& Y_h \, ,
	}
	\]
	where $d = \dim X- \dim  Y$. \qed
\end{lemma}

\begin{proof}[Proof of Proposition~\ref{Prop.Nice_Parametrization_closed_alg_subset}]
	By Corollary~\ref{Cor.Inductive-limit-topology}, there exists an integer $d$ with
	$Z \subseteq \pi_d(H_d)$. So we restrict $\pi_d$ to find an irreducible $W \subseteq H_d$ 
	and a closed morphism $\rho \colon W \to \Bir(X)$ with image equal to $Z$. 
	Up to replacing $W$ by a smaller closed irreducible subset, 
	we have that $\rho(A) \neq Z$ for all proper closed subsets $A \subsetneq W$.

	By Corollary~\ref{Cor.Decomp_of_morphism_weak} there is an open 
	dense subset $W' \subseteq W$, a dominant morphism $\lambda \colon W' \to U$,
	and an injective rationally universal 
	morphism $\eta \colon U \to \Bir(X)$ with $\rho |_{W'} = \eta \circ \lambda$.
	Note that general fibres of $\lambda$ are finite. Indeed, otherwise
there is a closed irreducible proper subset
	$W'' \subsetneq W'$ such that $\lambda(W'')$ is dense in $U$ 
	(e.g.~by Lemma~\ref{Lem.Kraft}); but this implies that $A \coloneqq \overline{W''}$
	is a proper closed subset of $W$ with $\rho(A) = Z$, contradiction.
	
	After shrinking $U$ (and replacing $W'$ by
	$\lambda^{-1}(U)$) we can assume that $\lambda$ is surjective, flat, and finite. Moreover,
	by assumption, $\rho(W \setminus W')$ is a proper closed subset of $Z$.
	After replacing $U$ by $\eta^{-1}(Z \setminus \rho(W \setminus W'))$ 
	(and $W'$ by $\lambda^{-1}(U)$) we obtain that $\eta(U) = \rho(W')$
	is open and dense in $Z$ and $W' = \rho^{-1}(\rho(W'))$. As
	$\rho \colon W \to Z$ is closed, this implies that
	\[
		\rho |_{W'} \colon W' = \rho^{-1}(\rho(W')) \xlongrightarrow{\lambda} U 
		\xlongrightarrow[\textrm{bij.}]{\eta} \rho(W')	= \eta(U)
	\]
	is closed as well. As a consequence, $\eta \colon U \to \eta(U)$ 
	is closed, and thus a homeomorphism.
\end{proof}

\begin{corollary}
	\label{Cor.Constructible_images}
	For every morphism to $\Bir(X)$, the image of a constructible set is again constructible. 
	In particular, every algebraic subset $Z \subseteq \Bir(X)$ contains a subset 
	that is open and dense in the closure $\overline{Z}$. 
\end{corollary}

\begin{proof}
	It is enough to show for an irreducible $V$
	and a morphism $\rho \colon V \to \Bir(X)$ that $\rho(V)$
	is constructible in $\Bir(X)$. We proceed by induction on $\dim V$, where
	the case $\dim V = 0$ is clear. Recall that $\overline{\rho(V)}$
	is an irreducible closed algebraic subset of $\Bir(X)$, see Corollary~\ref{Cor.Closure_alg_subset}.

	\begin{claim}
		\label{Claim.open_subset}
		There exists a subset $O \subseteq \rho(V)$ such that $O$ is dense and open in 
		$\overline{\rho(V)}$.
	\end{claim}

	\begin{proof}
		By Proposition~\ref{Prop.Nice_Parametrization_closed_alg_subset} we get
		an injective rationally universal morphism $\eta \colon U \to \Bir(X)$ such that 
		$\eta(U)$ is open and dense in $\overline{\rho(V)}$ and $\eta$
		induces a homeomorphism $U \to \eta(U)$. Note that the morphism
		\[
			\rho' \coloneqq \rho |_{\rho^{-1}(\eta(U))} \colon \rho^{-1}(\eta(U)) \to 
			\Bir(X)
		\]
		has as image the set $\rho(V) \cap \eta(U)$.
		By the rational universality of $\eta$, we find a dominant rational map
		$f \colon \rho^{-1}(\eta(U)) \dashrightarrow U$ with $\eta \circ f = \rho'$.
		Since $\rho(V) \cap \eta(U)$ is dense in $\eta(U)$ and $\eta \colon U \to \eta(U)$
		is a homeomorphism, it follows that $f$ is dominant. Hence,
		$U$ contains an open dense subset $U_0$ that lies in the image of $f$.
		Then $Z = \eta(U_0)$ is our desired subset.
	\end{proof}

	Let $O \subseteq \rho(V)$ be the subset of Claim~\ref{Claim.open_subset}.
	Then $V \setminus \rho^{-1}(O)$ is a proper, closed subset of $V$.
	By induction, $\rho(V \setminus \rho^{-1}(O)) = \rho(V) \setminus O$ is constructible in 
	$\Bir(X)$. As $O$ is constructible in $\Bir(X)$ the statement follows.
\end{proof}

\begin{corollary}
	\label{Cor.Image_morphism_finite_diml}
	If $\rho \colon  V \to \Bir(X)$ is a morphism and $V$ is an irreducible variety, then
	$\dim \overline{\rho(V)} \leq \dim V$. 
\end{corollary}

\begin{remark}
	\label{Rem.Image_morphism_finite_diml}
	The statement is much easier, in case $\rho(V)$ is closed in $\Bir(X)$. Indeed, the	
	argument is a variant of the proof of~\cite[Lemma~2.7]{KrReSa2021Is-the-affine-spac}:
	
	Let $\tau \colon W \to \Bir(X)$ be an injective morphism that has its image in $\rho(V)$. Let
	$F = \set{(v, w) \in V \times W}{\tau(w) = \rho(v)}$. By 
	Proposition~\ref{Prop.Prel}\eqref{Prop.Prel_diagonal}
	the set $F$ is closed in $V \times W$, and we have the following commutative diagram
	\[
		\xymatrix@R=15pt{
			F \ar[rr]^-{(v, w) \mapsto w} \ar[d]_-{(v,w) \mapsto v} &&  
			W \ar[d]^-{\tau}_-{\textrm{inj.}} \\
			V \ar[rr]_-{\rho}^-{\textrm{surj.}} && \rho(V) \, .
		}
	\]
	As $\rho$ is surjective, it follows that $F \to W$ is surjective. Since
	$\tau$ is injective, it follows that $F \to V$ is injective.
	This implies that $\dim W \leq \dim F \leq \dim V$, whence, $\dim \rho(V) \leq \dim V$.
\end{remark}

\begin{proof}[Proof of Corollary~\ref{Cor.Image_morphism_finite_diml}]
	By Corollary~\ref{Cor.Closure_alg_subset}
	$\overline{\rho(V)}$ is a closed, irreducible, algebraic subset of $\Bir(X)$. 
	There exists an algebraic family $\theta$ of birational transformations of $X$
	parametrized by an irreducible variety $W$ 
	such that $\rho_{\theta} \colon W \to \Bir(X)$ is a closed morphism with image
	$\overline{\rho(V)}$ and there is an open
	dense $W' \subseteq W$ such that $\rho_{\theta}(W')$ is open dense in $\overline{\rho(V)}$
	and $\rho_{\theta} |_{W'} \colon W' \to \rho(W')$
	decomposes into a surjetive finite morphism of varieties and a homeomorphism
	(see Proposition~\ref{Prop.Nice_Parametrization_closed_alg_subset}).

	Let $X_0 \subseteq X$ be an open dense subset in the projection to $X$
	of $\lociso(\theta) \cap \lociso(\vartheta)$, where  
	$\vartheta |_{V \times X_0}$ is the algebraic family of birational transformations of $X$ with $\rho_{\vartheta} = \rho$.
	Using Lemma~\ref{Lem.FinitelyManyPoints} we may find $n, m \geq 1$ and a Ramanujam 
	point $y = (y_1, \ldots, y_n) \in (X_0)^n$ for $\theta$ and a Ramanujam point 
	$z = (z_1, \ldots, z_m) \in (X_0)^m$ for $\vartheta |_{W \times X_0}$. This implies that 
	$(y_1, \ldots, y_n, z_1, \ldots, z_m) \in X^{n+m}$ is a Ramanujam point for 
	$\theta$
	and for $\vartheta$.
	Consider the rational maps
	\[
		\begin{array}{rcl}
		\lambda \colon W & \dashrightarrow & X^{n+m} \, , \\  
		w & \dashmapsto &  (\rho_{\theta}(w)(y_1), \ldots, \rho_{\theta}(w)(y_n), 
		\rho_{\theta}(w)(z_1), \ldots, \rho_{\theta}(w)(z_m)) 
		\end{array}
	\]
	and
	\[	
	    \begin{array}{rcl}
		\alpha \colon V & \dashrightarrow &  X^{n+m} \, , \\
		v & \dashmapsto & (\rho(v)(y_1), \ldots, \rho(v)(y_n), \rho(v)(z_1), \ldots, \rho(v)(z_m)) \, .
		\end{array}
	\]
	We may replace $V$ by $\dom(\alpha)$,
	since this does not change $\dim \overline{\rho(V)}$ and $\dim V$.
	Thus, $\alpha \colon V \to X^{n+m}$ is a morphisms.

	As $\rho(V) \cap \rho_{\theta}(W')$ is dense in $\rho_{\theta}(W')$
	and since $\rho_{\theta} |_{W'} \colon W' \to \rho(W')$ decomposes into a surjective 
	finite morphism and a homeomorphism, we get that
	$\rho_{\theta}^{-1}(\rho(V)) \cap W'$ is dense in $W'$. 
	In particular, $\rho_{\theta}^{-1}(\rho(V))$
	is dense in $W$.
	Then $\lambda(\rho_{\theta}^{-1}(\rho(V)) \cap \dom(\lambda))$ is dense in
	$\overline{\lambda(\dom(\lambda))}$, and it is contained in 
	$\overline{\alpha(V)}$, as $\alpha$ is a morphism.
	Moreover, 
	\[
		\lambda |_{W' \cap \dom(\lambda)} \colon 
		W' \cap \dom(\lambda) \to \overline{\lambda(\dom \lambda)}
	\] 
	is dominant and has finite fibres. In particular, 
	$\dim W = \overline{\lambda(\dom \lambda)}$.
	Hence, we get the following estimate
	\begin{eqnarray*}
		\dim \overline{\rho(V)} \stackrel{\textrm{Rem.~\ref{Rem.Image_morphism_finite_diml}}}{\leq} 
		\dim W = \dim \overline{\lambda(\dom(\lambda))}
		&=& \dim \overline{\lambda(\rho_{\theta}^{-1}(\rho(V)) \cap \dom(\lambda))} \\
		&\leq& \dim \overline{\alpha(V)} \leq \dim V \, .
	\end{eqnarray*}
	This implies the lemma.
\end{proof}

As a further consequence of Proposition~\ref{Prop.Nice_Parametrization_closed_alg_subset}
we can prove a fibre dimension formula for morphisms to $\Bir(X)$:

\begin{corollary}
	\label{Cor.Fibre-dimension-formula}
	Let $\rho \colon V \to \Bir(X)$ be a morphism
	with irreducible $V$.
	Then there is an open dense subset $U \subseteq V$ such that
	\[
		\dim_u (U \cap \rho^{-1}(\rho(u))) = \dim V - \dim \overline{\rho(V)}
		\quad \quad \textrm{for all $u \in U$} \, ,
	\]
	where $\dim_u$ denotes the local dimension at $u$.
\end{corollary}

\begin{proof}
	Take a Ramanujam point $y = (y_1, \ldots, y_n) \in X^n$ for the algebraic family
	associated to $\rho$.
	Let $U \subseteq V$ be an open dense subset with $U \times \{y_i\} \subseteq \lociso(\theta)$ 
	for all $i = 1, \ldots, n$ such that the morphism
	\[
		\eta \colon U \to X^n \, , \quad u \mapsto (\rho(u)(y_1), \ldots, \rho(u)(y_n))
	\]
	has a locally closed image and all fibres are equidimensional of 
	the same dimension $d$.
	By the definition of a Ramanujam point, 
	$\eta^{-1}(\eta(u)) = U \cap \rho^{-1}(\rho(u))$ for all $u \in U$.

	It remains to show that $\dim \overline{\rho(V)} + d = \dim V$. This follows if we show 
	that $\dim \overline{\rho(V)} = \dim \eta(U)$.
	After shrinking $U$ we may assume that there exists a closed irreducible subvariety
	$U_1 \subseteq U$ such that $\eta(U_1) = \eta(U)$ and 
	$\eta |_{U_1} \colon U_1 \to \eta(U_1)$ has finite fibres (see e.g.~Lemma~\ref{Lem.Kraft}).
	As $y$ is a Ramanujam point, 
	$\rho(U) = \rho(U_1)$ and $\rho |_{U_1} \colon U_1 \to \Bir(X)$ has finite fibres.
	Hence, there exists an injective morphism $\rho' \colon U_1' \to \Bir(X)$
	such that $\overline{\rho'(U_1')} = \overline{\rho(U_1)}$ and $\dim U_1' = \dim U_1$
	(by Corollary~\ref{Cor.Decomp_of_morphism_weak} applied to $\rho |_{U_1}$).
	As $\rho'$ is injective, $\dim U_1' \leq \dim \overline{\rho'(U_1')}$ and
	\[
		\dim \eta(U) = \dim \eta(U_1) = \dim U_1 = \dim U_1' 
		\xlongequal{\textrm{Cor.~\ref{Cor.Image_morphism_finite_diml}}} 
		\dim \overline{\rho'(U_1')} = \dim \overline{\rho(V)} \, . \qedhere
	\]
\end{proof}

\begin{remark}
	\label{Rem.Nice_parametrization}
	Using Corollary~\ref{Cor.Fibre-dimension-formula}, it follows that $\dim W = \dim Z$
	in Proposition~\ref{Prop.Nice_Parametrization_closed_alg_subset}. Indeed, with the notation
	of Proposition~\ref{Prop.Nice_Parametrization_closed_alg_subset} we have
	\[
		\dim Z \xlongequal{\textrm{Cor.~\ref{Cor.Fibre-dimension-formula}}}
		\dim W'	= \dim W \, .
	\]
\end{remark}

\begin{corollary}
	\label{Cor.Krull_dimensionI}
	If $Z_0 \subsetneq Z_1$ are closed irreducible algebraic subsets of $\Bir(X)$,
	then $\dim Z_0 < \dim Z_1$.
\end{corollary}

\begin{proof}
	Let $Z_0 \subsetneq Z_1$ be closed irreducible subsets of $\Bir(X)$.
	Let $\rho \colon W_1 \to \Bir(X)$ be a morphism with $\rho(W_1) = Z_1$
	as in Proposition~\ref{Prop.Nice_Parametrization_closed_alg_subset}. Hence, we get the 
	following estimate
	\[
		\dim Z_1 \xlongequal{\textrm{Rem.~\ref{Rem.Nice_parametrization}}} \dim W_1 > \dim \rho^{-1}(Z_0) 
		\stackrel{\textrm{Cor.~\ref{Cor.Image_morphism_finite_diml}}}{\geq} \dim Z_0 \, .
		\qedhere
	\]
\end{proof}

The corollary can be generalized to the fact that  for a closed algebraic subset of 
$\Bir(X)$ the Krull-dimenion induced by the topology on $\Bir(X)$
coincides with our definition of dimension for a subset of $\Bir(X)$:

\begin{corollary}
	\label{Cor.Krull_dimensionII}
	Let $Z \subseteq \Bir(X)$ be closed. Then 
	$\dim Z$ is the supremum over all $d$, where 
	$Z_0 \subsetneq Z_1 \subsetneq \ldots  \subsetneq Z_d \subseteq Z$
	is a chain of closed irreducible algebraic subsets of $Z$.
\end{corollary}

\begin{proof}
	Let $D \in \NN \cup \{\infty\}$ be the supremum over all $d$, where 
	$Z_0 \subsetneq Z_1 \subsetneq \ldots  \subsetneq Z_d \subseteq Z$
	is a chain of closed, irreducible, algebraic subsets of $Z$.

	``$D \leq \dim Z$'': If $Z_0 \subsetneq Z_1 \subsetneq \ldots  \subsetneq Z_d \subseteq Z$
	is a chain of closed, irreducible, algebraic subsets in $Z$, then 
	Corollary~\ref{Cor.Krull_dimensionI} implies that $\dim Z \geq d$.

	``$\dim Z \leq D$'': Let $A \subseteq Z$ be an irreducible closed algebraic subset.
	It is enough to show that $\dim A \leq D$ 
	(by using Corollary~\ref{Cor.Image_morphism_finite_diml}). 
	There exists an irreducible $U$
	and an injetive morphism $\eta \colon U \to \Bir(X)$ such that $\eta(U)$ is open and dense in $A$
	and $\eta$ restricts to a homeomorphism $U \to \eta(U)$. Let $d = \dim U = \dim A$ 
	(cf.~Corollary~\ref{Cor.Image_morphism_finite_diml}). Then there exists a chain
	of irreducible closed subsets $U_0 \subsetneq U_1 \subsetneq \cdots \subsetneq 
	U_d \subseteq U$. Hence, $A_0 \subsetneq A_1 \subsetneq \ldots  \subsetneq A_d \subseteq A$
	is a chain of irreducible closed algebraic subsets of $A$,
	where $A_i = \overline{\eta(U_i)}$ and thus $\dim A = d \leq D$.
\end{proof}

There are two easy applications:

\begin{corollary}
	\label{Cor.Parametr_finite_diml_group}
	If $G \subseteq \Bir(X)$ is a connected closed subgroup of finite
	dimension, then $G$ is an irreducible algebraic subset of $\Bir(X)$.
\end{corollary}

\begin{proof}
	There exists an ascending
	exhausting chain of closed irreducible algebraic susbets in $G$ 
	(see Corollary~\ref{Cor.Chain_irreducible_alg_subsets}) and by Corollary~\ref{Cor.Krull_dimensionII} this
	chain becomes eventually stationary.
\end{proof}

It would be interesting to find an example of an irreducible closed subset of 
finite dimension in $\Bir(X)$ that is not algebraic.

\begin{corollary}
	\label{Cor.algebraic_subset_group_is_closed}
	Let $G \subseteq \Bir(X)$ be a subgroup that is also an algebraic subset.
	Then $G$ is closed in $\Bir(X)$ and for every open dense subset $U \subseteq G$ we have
	$U \circ U = G$.
\end{corollary}

\begin{proof}
	Let $U \subseteq G$ be an open dense subset.
	By Corollary~\ref{Cor.Constructible_images} we may shrink $U$
	such that it is open and dense in the closure $\overline{G}$. Since $\overline{G}$
	is an algebraic subset of $\Bir(X)$ (see Corollary~\ref{Cor.Closure_alg_subset}) 
	and multiplication by every element of $\overline{G}$
	is a homeomorphism of $\overline{G}$, it follows that $\overline{G} = U \circ U \subseteq G$.
\end{proof}

As a further application of Corollary~\ref{Cor.Krull_dimensionII}, 
we study the maximal irreducible closed subsets of a closed finite-dimensional subset in $\Bir(X)$:

\begin{corollary}
	\label{Lem.Irred_comp_of_finite_diml_close_subset}
	Let $Z \subseteq \Bir(X)$ be closed and $\dim Z < \infty$. Then $Z$
	contains at most countably many maximal irreducible closed algebraic subsets
	and $Z$ is the union of them.
\end{corollary}

\begin{proof} 
	Let $\mathcal{M}$ be the set of all irreducible components of all Noetherian spaces 
	$Z \cap \pi_d(H_d)$, $d \geq 1$, see Corollary~\ref{Cor.Inductive-limit-topology}.
	The elements of $\mathcal{M}$ are irreducible closed algebraic subsets of $\Bir(X)$.
	Let $\mathcal{M}_0$ be the subset of the maximal elements in $\mathcal{M}$ under inclusion.
	Since $\dim Z < \infty$, it follows from Corollary~\ref{Cor.Krull_dimensionII} that every element in 
	$\mathcal{M}$
	is contained in an element of $\mathcal{M}_0$ and hence the union of the elements in $\mathcal{M}_0$ is equal to $Z$.
	The set $\mathcal{M}_0$ is countable, as $\mathcal{M}$ is countable.
\end{proof}

\section{\texorpdfstring{Subgroups of $\Bir(X)$ parametrized by varieties}
{Subgroups of Bir(X) parametrized by varieties}}

The goal of this section is to prove that a closed 
finite-dimensional subgroup of $\Bir(X)$ with finitely many connected
components
admits a unique structure of an algebraic group. In case $X = \PP^n$ 
this is proven in~\cite[Corollary~2.18]{BlFu2013Topologies-and-str}. This strategy 
is not applicable in the general setting. Our proof
rather uses the ideas from~\cite{Ra1964A-note-on-automorp}, where a similar result
is proved for $\Aut(X)$.

\begin{proposition}
	\label{Prop.Ramanujam_gen}
	Let $G \subseteq \Bir(X)$ be a closed finite-dimensional subgroup
	with finitely many connected components.
	Then there exists an algebraic group $H$ and
	a rationally universal morphism $\iota \colon H \to \Bir(X)$ that 
	restricts to a group isomorphism $H \to G$, which is also a homeomorphism. 
	Moreover, $\iota \colon H \to \Bir(X)$ satisfies the following universal property:

	\hypertarget{Eq.Universal_Property_groups}{\textnormal{($\ast$)}}
	If $H'$ is an algebraic group and 
	$\rho \colon H' \to \Bir(X)$ is a morphism with image in $G$ that is also a group homomorphism,
	then there exists a unique homomorphism of algebraic groups $f \colon H' \to H$ such that 
	$\rho = \iota \circ f$.
\end{proposition}

\begin{lemma}
	It is enough to prove Proposition~\ref{Prop.Ramanujam_gen} 
	for connected $G$.
\end{lemma}

\begin{proof}
	We assume that there is a rationally universal morphism $\iota_0 \colon H_0 \to \Bir(X)$
	that restricts to a group isomorphism $H_0 \to G^\circ$, where $G^\circ$
	denotes the connected component of the identity of $G$, and $\iota_0$ 
	satisfies the universal property~\hyperlink{Eq.Universal_Property_groups}{\textnormal{($\ast$)}}
	from Proposition~\ref{Prop.Ramanujam_gen}.
	In particular, for all $g \in G$ there exists an automorphism $c_g$ of the algebraic group 
	$H_0$ such that $\iota_0(c_g(h)) = g^{-1} \circ \iota_0(h) \circ g$
	for all $h \in H_0$.

	Let $g_1, \ldots, g_m \in G$ be representatives of the cosets of $G_0$ in $G$
	and define 
	\[
		\iota \colon H \coloneqq \coprod_{i = 1}^m H_0^{(i)} \to \Bir(X) \quad
		\textrm{via} \quad
		\iota |_{H_0^{(i)}} \colon H_0^{(i)} \coloneqq H_0 \xrightarrow[]{\iota_0} G^\circ 
		\xrightarrow[]{g \mapsto g_i g} g_i G^\circ \, .
	\]
	Then $\iota \colon H \to \Bir(X)$ is a rationally universal 
	morphism that restricts to a homeomorphism
	$H \to G$. Moreover, we endow $H$ with the group structure such that $\iota$ becomes a group
	isomorphism. 
	
	\begin{claim}
		The group $H$ is an algebraic group.
	\end{claim}

	\begin{proof}
	Indeed, for $1 \leq i, j \leq m$
	there exists a unique integer  $1 \leq k \leq m$ such that 
	$g_i G^\circ g_j G^{\circ} = g_k G^\circ$ inside $G$. In particular, there
	exists a unique $g_0 \in G^\circ$ with $g_i g_j = g_k g_0$.
	Let $h_0 \in H_0$ be the preimage of $g_0$ under $\iota_0$.
	Now, the multiplication map $H \times H \to H$ restricts to the morphism
	\[
		H_0^{(i)} \times H_0^{(j)} \to H_0^{(k)} \, , \quad 
		(h, h') \mapsto h_0 c_{g_j}(h) h' \, ,
	\]  
	and hence, the multiplication map
	$H \times H \to H$ is a morphism. 
	
	Similarly, for $1 \leq i \leq m$ there exists a unique
	$1 \leq l \leq m$ such that $(g_i G^\circ)^{-1} = g_l G^\circ$ and hence, we may choose
	$h_1 \in H_0$ with $g_i^{-1} = g_l \iota_0(h_1)$. Now, the inversion map
	$H \to H$ restricts to the morphism 
	\[
		H_0^{(i)} \to H_0^{(l)} \, , \quad h \mapsto h_1 c_{g_i^{-1}}(h^{-1}) \, ,
	\] 
	and hence, the inversion map $H \to H$ is an automorphism 
	of algebraic groups.
	\end{proof}

	If $\rho \colon H' \to \Bir(X)$ is a morphism that is also a group homomorphism,
	then there exists a unique group homomorphism 
	$f \colon H' \to H$ with $\rho = \iota \circ f$.
	By assumption, the restriction 
	$f |_{\rho^{-1}(G^\circ)} \colon \rho^{-1}(G^\circ) \to H$ is a homomorphism
	of algebraic groups. This implies that $f$ is a homomorphism of algebraic groups.
\end{proof}

From now on we assume that $G$ is connected.
For the proof of Proposition~\ref{Prop.Ramanujam_gen}, 
we take an injective rationally universal morphism 
$\eta \colon U \to \Bir(X)$ that induces 
a homeomorphism onto an open dense subset of $G$
(see Proposition~\ref{Prop.Nice_Parametrization_closed_alg_subset} and Corollary~\ref{Cor.Parametr_finite_diml_group}).
Moreover, $\dim G = \dim U$ (see e.g.~Corollary~\ref{Cor.Fibre-dimension-formula}).
We will identify $U$ with its (open) image in $G$ under $\eta$ and hence
$U \circ U = G$ (see Corollary~\ref{Cor.algebraic_subset_group_is_closed}).

Denote by $\kappa$ the algebraic family of birational transformations of 
$\Bir(X)$ para\-metrized by $U$ associated to $\eta$, 
and let $p = (p_1, \ldots, p_n) \in X^n$ be a Ramanujam point for $\kappa$. 
After shrinking $U$ we may assume that
$U$ is smooth and affine, and $U \times \{p\} \subseteq \lociso(\kappa_n)$,
where $\kappa_n$ denotes the diagonal family on $X^n$ induced by $\kappa$.
Hence,
\[
	 \alpha \colon U \to X^n \, , \quad u 
	 \mapsto (u(p_1), \ldots, u(p_n))
\]
is an injective morphism. Note that all these properties are 
preserved if we pass to an open dense subset of $U$, which we will frequently do.
\begin{lemma}
	\label{Lem.Gen_immersion}
	For general $u \in U$, the differential 
	$\textrm{d}_{u} \alpha \colon T_{u} U \to T_{\alpha(u)} X^n$
	of $\alpha$ at $u$ is injective.
\end{lemma}

\begin{proof}
	The statement is true in case $\car(\kk) = 0$ (e.g.~by Zariski's main theorem, 
	see~\cite[Corollary~12.88]{GoWe0Algebraic-geometryI}) 
	and hence, we may assume that $\car(\kk) > 0$.
	Let $\ker(\textrm{d} \alpha) \subseteq T U$ be the kernel in the tangent bundle $TU$ 
	of the differential 
	$\textrm{d} \alpha \colon TU \to TX^n$. 
	After shrinking $U$ we may assume that $\ker(\textrm{d} \alpha)$
	is a sub-bundle of $TU$. One shows that $\ker(\textrm{d} \alpha)$ is an integrable
	sub-bundle of $TU$ in the sense of \cite[\S3]{Se1958Loperation-de-Cart}, 
	i.e.~$\ker(\textrm{d} \alpha)$ is a restricted Lie-subalgebra of $TU$; the conditions
	are directly seen to be satisfied if we interpret vector fields on smooth affine varieties
	as derivations of the coordinate ring. 
	
	Fix $u_1 \in U$. 
	By \cite[Théorème~2, Proposition~7]{Se1958Loperation-de-Cart} there is a bijective 
	morphism $\xi \colon U \to U'$ to a smooth, irreducible variety $U'$ 
	such that the kernel of $\textrm{d} \xi$ 
	is equal to $\ker(\textrm{d} \alpha)$ and there exist
	$y_1, \ldots, y_m \in \OO_{U, u_1} \subseteq \OO_{U \times X^n, (u_1, p)}$ that form a 
	$\car(\kk)$-basis
	of $\OO_{U \times X^n, (u_1, p)}$ over $\OO_{U' \times X^n, (u_1', p)}$ (where $u_1' = \xi(u_1)$)
	and there exist $\OO_{U' \times X^n, (u_1', p)}$-derivations 
	$\frac{\partial}{\partial y_1}, \ldots, \frac{\partial}{\partial y_m}$
	of $\OO_{U \times X^n, (u_1, p)}$ such that $\frac{\partial y_j}{\partial y_i}$
	is equal to the Kronecker delta for all  $(i, j)$. In particular, the elements
	$y_1^{\rho_1} \cdots y_m^{\rho_m}$ where $(\rho_1, \ldots, \rho_m)$ runs through all 
	indices of $\{0, \ldots, \car(\kk)-1\}^m$
	form a basis of the
	$\OO_{U' \times X^n, (u_1', p)}$-module $\OO_{U \times X^n, (u_1, p)}$ 
	(see e.g.~\cite[Remark~15.1]{Ku1986Kahler-differentia}) and hence,
	\begin{equation}
		\label{Eq.Intersection}
		\bigcap_{i = 1}^m \ker\left(\frac{\partial}{\partial y_i}\right) = 
		\OO_{U' \times X^n, (u_1', p)} \, .
	\end{equation}
	This implies that $\frac{\partial}{\partial y_1}, \ldots, \frac{\partial}{\partial y_m}$ are linearly independent elements in $T_{u_1} U \times T_p X^n$ contained in the kernel of
	\[
		\textrm{d}_{(u_1, p)}(\xi \times \id_{X^n}) \colon 
		T_{(u_1, p)}(U \times X^n) \to T_{(u_1', p)}(U' \times X^n) \, .
	\]
	Note that the kernel of the above linear map is given by
	\begin{eqnarray*}
		\ker(\textrm{d}_{u_1} \xi) \times \{0\} &=& \ker(\textrm{d}_{u_1} \alpha)  \times \{0\} \\
		&=&
		\ker( \textrm{d}_{(u_1, p)}(\pr_2 \circ\kappa_n |_{U \times \{p\}})) \\
		&=& \ker(\textrm{d}_{(u_1, p)}(\pr_2 \circ\kappa_n ))
		\cap T_{u_1} U  \times \{ 0 \} \subseteq T_{u_1} U \times T_p X^n  \, ,
	\end{eqnarray*}
	where $\pr_2 \colon U \times X^n \to X^n$ denotes the projection to the second factor.
	Hence, 
	\[
		\frac{\partial (\pr_2 \circ \kappa_n)^\ast(f) }{\partial y_i} = 0
		\quad \quad
		\textrm{for all $f \in \OO_{X^n, u_1(p)}$ and $i = 1, \ldots, m$} \, .
	\]
	Using~\eqref{Eq.Intersection},
	$(\pr_2 \circ \kappa_n)^\ast(f) \in \OO_{U' \times X^n, (u_1', p)}$
	for all $f \in \OO_{X^n, u_1(p)}$. Hence, there exists a rational
	map $\mu \colon U' \times X^n \dashrightarrow X^n$ such that 
	$\pr_2  \circ \kappa_n = \mu \circ (\xi \times \id_{X^n})$. By restriction to $U \times X$,
	$U' \times X$ and $X$ (where we embed $X$ diagonally into $X^n$) we get a rational self-map 
	$\kappa'$ of $U' \times X$ such that 
	$\kappa' \circ (\xi \times \id_X) = (\xi \times \id_X) \circ \kappa$. As $\kappa$ is birational,
	it follows that $\kappa'$ is birational as well and by further shrinking $U$ we may assume
	that $\kappa'$ is an algebraic family of birational transformations of $\Bir(X)$
	parametrized by $U'$. By construction, we get now
	$\rho_{\kappa'} \circ \xi = \eta \colon U \to \Bir(X)$. Since $\eta$ is an injective rationally 
	universal morphism, it follows that $\xi \colon U \to U'$ is birational. Using that
	$\xi$ is a bijective morphism, we conclude that $\xi$ is an isomorphism by Zariski's main theorem.
	This shows that $\ker(\textrm{d} \alpha) = 0$ and thus
	$\textrm{d}_u \alpha$ is injective for all $u \in U$.
\end{proof}

\begin{proof}[Proof of Proposition~\ref{Prop.Ramanujam_gen}]
	We use the setup introduced before Lemma~\ref{Lem.Gen_immersion}.
	Using Lemma~\ref{Lem.Gen_immersion} we may shrink $U$ such that
	$\alpha$ becomes an open embedding $U \to \overline{\alpha(U)}$. 
	Fix $u_0 \in U$ with $(u_0, p) \in \lociso(\kappa_n)$. 
	By post composing $\eta$ with multiplication 
	by $u_0^{-1} \in G$ (and possibly shrinking $U$ further) 
	we may assume that $\kappa_n(u_0, p) = (u_0, p)$.

	\begin{claim}
		\label{Claim.Restriction}
		$\kappa_n \colon U \times X^n \dashrightarrow U \times X^n$ 
		restricts to a birational self-map of $U \times \alpha(U)$.
	\end{claim}

	\begin{proof}
	Note that $\lociso(\kappa_n)$ has a non-trivial intersection with 
	$U \times \alpha(U)$
	(both contain $(u_0, p)$). 
	Let $\varepsilon \colon U \times \alpha(U) \to G$ be defined
	by $\varepsilon(u, q) = u \circ \alpha^{-1}(q)$. As $\varepsilon$ is continuous and
	$\varepsilon$ is surjective (since $U \circ U = G$), 
	it follows that $\varepsilon^{-1}(U)$ is a dense open subset of 
	$U \times \alpha(U)$. For $(u_1, \alpha(u_2)) \in \varepsilon^{-1}(U) \cap \lociso(\kappa_n)$
	we get $u_1 \circ u_2 \in U$ and hence
	\[
		\kappa_n(u_1, \alpha(u_2)) = \kappa_n(u_1, (u_2(p_1), \ldots, u_2(p_n)))
		= (u_1, (u_1 \circ u_2)(p_1), \ldots, (u_1 \circ u_2)(p_n))
	\] 
	is contained in $U \times \alpha(U)$. This shows the claim.
	\end{proof}

	By Claim~\ref{Claim.Restriction},
	\[
		\varphi \colon U \times U \xrightarrow[\sim]{\id_U \times \alpha} U \times \alpha(U) 
		\stackrel{\kappa_n}{\dashrightarrow} 
		U \times \alpha(U) \xrightarrow[\sim]{\id_U \times \alpha^{-1}} U \times U
	\]
	is a birational $U$-map. Let
	\[
		L \coloneqq 
		\lociso(\varphi) \cap \set{(u_1, u_2) \in U \times U}{u_1 \circ u_2 \in U}
		\subseteq U \times U
	\]
	and
	\[
		L^{-1} \coloneqq 
		\lociso(\varphi^{-1}) \cap \set{(u_1, u_2) \in U \times U}{u_1^{-1} \circ u_2 \in U}
		\subseteq U \times U
	\]
	Then $L$, $L^{-1}$ are open dense subsets of $U \times U$, 
	$\varphi \colon L \to L^{-1}$ is an isomorphism and
	\[
		\varphi(u_1, u_2) = (u_1, u_1 \circ u_2) \, ,	\quad
		\varphi^{-1}(u_1, u_2) = (u_1, u_1^{-1} \circ u_2) \, .
	\]
	Moreover, $U \dashrightarrow U$, $u \dashmapsto u^{-1}$
	is a birational transformation. Indeed, fix $u_0 \in U$ such that $U \times \{ u_0 \}$ intersects 
	$L^{-1}$ and $U \times \{ u_0^{-1} \}$ intersects $L$. Then $u \dashmapsto u^{-1}$
	is the composition of the dominant rational maps $U \dashrightarrow U$ given by 
	$u \dashmapsto u^{-1} \circ u_0$ and $u \dashmapsto u \circ u_0^{-1}$, respectively.
	Since $u \dashmapsto u^{-1}$ is an involution, we conclude that it is birational.

	Hence, there exists an open dense subset $M \subseteq U \times U$, such that
	\[
		M \to U \, , \quad (u_1, u_2) \mapsto u_2 \circ u_1^{-1}
	\]
	is a morphism.

	\begin{claim}
		\label{Claim.Group_chunk}
		There is a connected algebraic group $H$ and a birational map 
		$i \colon U \dashrightarrow H$ such that
		for general $(u_1, u_2) \in U \times U$ we have that $i(u_1) i(u_2) = i(u_1 \circ u_2)$.
	\end{claim}
	
	\begin{proof}
		Let $W$ be the locally closed subset of $U^3$ defined by:
		\[
			W = \Bigset{(u_1, u_2, u_3) \in U^3}{
							\begin{array}{l}			
								u_1 \circ u_2 = u_3, \, (u_1, u_2) \in L, \\
								(u_1, u_3) \in L^{-1}, \, (u_2, u_3) \in M
							\end{array} } \, .
		\]
		Note that $W$ is the intersection of the graphs of the morphisms
		\[
			\begin{array}{lclcl}
			L &\to& U \, , \quad (u_1, u_2) &\mapsto& u_1 \circ u_2 \\
			L^{-1} &\to& U \, , \quad (u_1, u_3) &\mapsto& u_1^{-1} \circ u_3 \\
			M &\to& U \, , \quad (u_2, u_3) &\mapsto& u_3 \circ u_2^{-1} \, .
			\end{array}
		\]
		Hence, for every $i = 1, 2, 3$, the projection
		$p_i \colon W \to U^2$, where
		the $i$-th factor is omitted, is an open embedding.
		Moreover, associativity holds, i.e. if $u_1, u_2, u_3 \in U$ and
		\[
			\begin{array}{rcl}
			(u_1, u_2) \, ,  \ (u_2,  u_3)\, ,  \ (u_1 \circ u_2, u_3) \, , \ 
			(u_1, u_2 \circ u_3) &\in& L
			\end{array}
		\]
		then $(u_1 \circ u_2) \circ u_3 = u_1 \circ (u_2 \circ u_3)$. As $\kk$ is algebraically closed,
		it follows from \cite[Proposition~3.2, Remarques~3.1(b), Exp.~XVIII]{DeGr1970Schemas-en-groupes} that there exist open dense subsets $U' \subseteq U$ and
		$W' \subseteq W \cap (U')^3$ such that $(U', W')$ is a group germ in the sense
		of \cite[D\'efinition~3.1, Exp.~XVIII]{DeGr1970Schemas-en-groupes}. 
		Now, ~\cite[Proposition~3.6, Th\'eor\`eme~3.7, Corollaire 3.13, Exp.~XVIII]{DeGr1970Schemas-en-groupes} imply that
		there exists an open embedding $i \colon U' \to H$ into a connected algebraic group $H$ such that
		for all $(u_1, u_2) \in p_3(W') \subseteq (U')^2$ we have that 
		$i(u_1) i(u_2) = i(u_1 \circ u_2)$.
		Hence, the claim follows.
	\end{proof}
		
	By further shrinking $U$ we may and will assume that 
	$i \colon U \to H$ is an everywhere defined dominant open embedding.
	Hence, $\kappa$ extends via 
	$i \times \id_X \colon U \times X \to H \times X$ to a birational transformation
	\[
		\vartheta \colon H \times X \dashrightarrow H \times X \, .
	\]
	As $H$ is a connected algebraic group, it follows that $i(U) i(U) = H$.
	
	\begin{claim}
		The rational map $\alpha \coloneqq \pr_X \circ \vartheta \colon G \times X \dashrightarrow X$
		defines a rational $G$-action, where $\pr_X$ denotes the projection to $X$, and 
		$\rho_{\vartheta} \colon H \to \Bir(X)$ is a group homomorphism with image $G$.
	\end{claim}

	\begin{proof}
		Let $\beta \coloneqq \pr_X \circ \kappa \colon U \times X \dasharrow X$. Then we have for general 
		$x \in X$ and $(u_1, u_2) \in L$ that $\beta(u_1, \beta(u_2, x)) = \beta(u_1 \circ u_2, x)$.
		Since $i(u_1 \circ u_2) = i(u_1)i(u_2)$ for general $(u_1, u_2) \in U \times U$,
		it follows that $\alpha(h_1, \alpha(h_2, x)) = \alpha(h_1 h_2, x)$ for general $x \in X$
		and general  $(h_1, h_2) \in H$. Now, Corollary~\ref{Cor.Rational_action_is_family} implies
		that $\vartheta$ is an algebraic family and $\rho_{\vartheta} \colon H \to \Bir(X)$ is a 
		group homomorphism. The last statement in the claim follows from
		$\rho_{\vartheta}(H) = \rho_{\vartheta}(i(U)i(U)) = 
		\rho_{\vartheta}(i(U)) \circ \rho_{\vartheta}(i(U)) = 
		U \circ U = G$.
	\end{proof}

By Proposition~\ref{Prop.Prel}\eqref{Prop.Prel_diagonal} the kernel $K$ of $\rho_{\vartheta} \colon H \to \Bir(X)$
is a closed normal subgroup of $H$.
Let $k \in K$. As $k i(U)$ and $i(U)$ are both open and dense in $H$,
there exist $u_1, u_2 \in U$ with $k i(u_1) = i(u_2)$. Hence, 
$u_1 = \rho_{\vartheta}(k i(u_1)) = \rho_{\vartheta}(i(u_2)) = u_2$. Consequently, $k$ is trivial
and thus $\rho_\vartheta \colon H \to \Bir(X)$ is injective. 
As $\eta(U)$ is open in $G$ and $\rho_{\vartheta} \circ i = \eta \colon U \to \eta(U)$ is a homeomorphism, we get that the restriction
$\rho_{\vartheta} |_{i(U)} \colon i(U) \to G$ is an open embedding. Hence,
the group isomorphism $\iota \coloneqq \rho_\vartheta \colon H \to G$ is a homeomorphism. 
Recall that 
\[
	\eta \colon U \xrightarrow{i} H \xrightarrow{\iota} \Bir(X)
\]
is a rationally universal morphism and hence, $\iota$ is a rationally universal morphism.

\medskip

	Now, let $\rho \colon H' \to \Bir(X)$ be a morphism that is also a group homomorphism.
	As $\iota$ is a rationally universal 
	injective group homomorphism, there exist a unique group homomorphism 
	$f \colon H' \to H$ with $\rho = \iota \circ f$ and a dense open subset
	$V \subseteq H'$ such that $f|_{V}$ is a morphism.  
	This implies that $f$ is a morphism.
\end{proof}

\section{Birational transformations preserving a fibration}\label{Sec.fibration}

Let $\pi \colon X \to Y$ be a dominant morphism of irreducible varieties.
The goal of this section is to study birational transformations
preserving the general fibres of $\pi$.

A birational transformation $\varphi$ of $X$ \emph{preserves the fibres of $\pi$} if there
is an open dense subset $U \subseteq \lociso(\varphi)$ such that $\varphi$ maps every fibre of
$U \to Y$ into a fibre of $\varphi(U) \to Y$.
In this case, $\varphi$ preserves all fibres of $\lociso(\varphi) \to Y$,
as the subset of those $(x_1, x_2) \in \lociso(\varphi) \times_Y \lociso(\varphi)$
with $\pi(\varphi(x_1)) = \pi(\varphi(x_2))$ is closed in 
$\lociso(\varphi) \times_Y \lociso(\varphi)$ and contains the open dense subset
$U \times_Y U$. Let
\begin{align*}
	\Bir(X, \pi\fib) &= \set{\varphi \in \Bir(X)}{\textrm{$\varphi$ preserves the fibres of $\pi$} } \, , \\
	\Bir(X, \pi) &= \set{\varphi \in \Bir(X)}{
	\textrm{there exists $\bar{\varphi} \in \Bir(Y)$ with $\bar{\varphi} \circ \pi = \pi \circ \varphi$} 
	} \, , \\
	\Bir(X/Y) &= \set{\varphi \in \Bir(X)}{\pi = \pi \circ \varphi} \, . 	
\end{align*}
Note that in general the inclusion $\Bir(X, \pi) \subseteq \Bir(X, \pi\fib)$ is proper:
Let $\car(\kk) = p > 0$ and take $\pi \colon \AA^2 \to \AA^2$, 
$(x,y) \mapsto (x^p, y)$. Then, for example, the isomorphism of $\AA^2$ that exchanges both factors
does not descend to a birational transformation of $\AA^2$.
However, let us define the following Property~\hyperlink{Eq.NiceFibration}{($\ast$)}
for the morphism $\pi$, which will turn out to be a sufficient condition for equality to hold.

\begin{quote}
	 \hypertarget{Eq.NiceFibration}{($\ast$)}
	 For a normal proper $\kk(Y)$-birational model $Z$ of 
	 the generic fibre of $\pi \colon X \to Y$, we have:
	 The finite field extension $\kk(Y) \subseteq \OO_Z(Z)$ is separable.\footnote{The ring of global functions
	 $\OO_Z(Z)$ is always a finite field extension of $\kk(Y)$, see e.g.~\cite[Theorem~12.65]{GoWe0Algebraic-geometryI}.}
\end{quote}
Note that Property~\hyperlink{Eq.NiceFibration}{($\ast$)} is independent of the choice of a 
normal proper $\kk(Y)$-birational model of the 
generic fibre of $\pi$. Indeed, a birational map between normal proper irreducible varieties over a field induces an isomorphism
between the fields of global functions, see~e.g.~\cite[Theorem~12.60]{GoWe0Algebraic-geometryI}.
Property~\hyperlink{Eq.NiceFibration}{($\ast$)} has the following geometrical interpretation:

\begin{lemma}
	\label{Lem.Stein_factorization}
	Assume that $X$ is normal, that $\pi \colon X \to Y$ is proper, and let $X \to Y' \to Y$ be its Stein factorization.
	Then $\pi$ satisfies Property~\hyperlink{Eq.NiceFibration}{($\ast$)} if and only if
	$Y' \to Y$ is generically \'etale, i.e.~$\kk(Y) \subseteq \kk(Y')$ is separable.
\end{lemma}

\begin{proof}
	Since $X$ is normal, the generic fibre $Z$ of $\pi$ is normal as well. As the  pull-back of the Stein factorization
	of $\pi$ is again the Stein factorization of $Z \to \Spec(\kk(Y))$ (see \cite[Remark~24.48]{GoWe0Algebraic-geometryII}), 
	the statement follows.
\end{proof}

Note that Lemma~\ref{Lem.Stein_factorization} implies in particular 
that Property~\hyperlink{Eq.NiceFibration}{($\ast$)} holds for $\pi$ in the following situations:
\begin{enumerate}[left=0pt]
	\item $\car(\kk) = 0$;
	\item $\pi$ is proper, $X$ is normal, and $\pi_{\ast} \OO_X = \OO_Y$;
	\item $\pi$ is generically \'etale;
	\item $\kk(Y)$ is inseparably closed in $\kk(X)$, 
	i.e.~if $\kk(Y) \subseteq L \subseteq \kk(X)$ is an intermediate field such that 
	$\kk(Y) \subseteq L$ is finite, then $\kk(Y) \subseteq L$
	is separable,
\end{enumerate}
where for the last two situations we look at the normalization of a completion of $\pi$.
Note that the last situation also covers the cases when $\kk(Y) \subseteq \kk(X)$ is separable (but not necessarily algebraic) or when the geometric generic fibre of $\pi$ is integral, see
\cite[Proposition~5.51]{GoWe0Algebraic-geometryI}.
Property~\hyperlink{Eq.NiceFibration}{($\ast$)} is preserved in the following situations:

\begin{lemma} $ $
	\label{Lem.NiceFibration_permanence}
	\begin{enumerate}[left=0pt]
		\item 	
		\label{Lem.NiceFibration_permanence1}
		Let $\pi' \colon X' \to Y'$ be a dominant morphism of irreducible varieties
		such that there exist birational maps $\varphi \colon X' \dashrightarrow X$,
		$\psi \colon Y' \dashrightarrow Y$ with $\pi \circ \varphi = \psi \circ \pi'$.
		Then $\pi$ satisfies~\hyperlink{Eq.NiceFibration}{($\ast$)} if and only if
		$\pi'$ satisfies~\hyperlink{Eq.NiceFibration}{($\ast$)}.
		\item 
		\label{Lem.NiceFibration_permanence2}
		If $\pi$ satisfies~\hyperlink{Eq.NiceFibration}{($\ast$)} and $V$
		is an irreducible normal variety, then 
		$\id_V \times \pi  \colon V \times X \to V \times Y$
		satisfies~\hyperlink{Eq.NiceFibration}{($\ast$)} as well.
	\end{enumerate}
\end{lemma}

\begin{proof}
	\eqref{Lem.NiceFibration_permanence1}:
	It is enough to consider the case where $\psi$ and $\varphi$ are open embeddings.
	However, in this case it is enough to note that $\kk(Y') = \kk(Y)$ and that the
	generic fibre of $\pi'$ is an open dense subset of the generic fibre of $\pi$.
	
	\eqref{Lem.NiceFibration_permanence2}: Using~\eqref{Lem.NiceFibration_permanence1}
	we may replace $\pi$ by the normalization of a completion of $\pi$. As the Stein factorization
	is compatible with flat base-change (see Remark~\cite[Remark~24.48]{GoWe0Algebraic-geometryII}),
	the statement follows from Lemma~\ref{Lem.Stein_factorization}.
\end{proof}

\begin{lemma}
	\label{Lem.Bir_Y_X_closed}
	The subgroups $\Bir(X, \pi\fib)$ and $\Bir(X/Y)$ are closed in $\Bir(X)$.
\end{lemma}

\begin{proof}
	Let $\theta$ be an algebraic family of birational transformations of $X$ parametrized by some variety $V$. For $(x_1, x_2) \in X \times_Y X$, consider the following subset of $V$:
	\[
		V_{x_1, x_2} = 
		\Bigset{v \in V}{ 
		\begin{array}{l}
		(v, x_1), (v, x_2) \in \lociso(\theta) \implies \\
		\pi(\pr_2(\theta(v, x_1))) = \pi(\pr_2(\theta(v, x_2)))
		\end{array}
		} \, ,
	\]
	where $\pr_2 \colon V \times X \to X$ denotes the natural projection.
	Note that the intersection of $V_{x_1, x_2}$ with the open subset
	\[
		U_{x_1, x_2} \coloneqq \set{v \in V}{(v, x_1), (v, x_2) \in \lociso(\theta)}	
	\]
	is closed in $U_{x_1, x_2}$ and since $V_{x_1, x_2}$
	contains the complement of $U_{x_1, x_2}$ in $V$, we get that 
	$V_{x_1, x_2}$ is closed in $V$. However, by definition
	\[
		\bigcap_{(x_1, x_2) \in X \times_Y X} V_{x_1, x_2} = \set{v \in V}{\rho_{\theta}(v) \in \Bir(X, \pi\fib)} \subseteq V \, ,
	\]
	and thus $\Bir(X, \pi\fib)$ is closed in $\Bir(X)$.
	Using a similar argument, where $V_{x_1, x_2}$, $U_{x_1, x_2}$ are replaced by
	$V_x = \set{v \in V}{(v, x) \in \lociso(\theta) \implies \pi(\pr_2(\theta(v, x))) = \pi(x)}$ and
	$U_x = \set{v \in V}{(v, x) \in \lociso(\theta)}$ shows that $\Bir(X/Y)$ is closed in $\Bir(X)$.
\end{proof}

As a direct consequence, we get that computing the invariants of a set and its closure is the same:

\begin{corollary}
	\label{Cor.Closure_subset_invariants}
	For every subset $S \subseteq \Bir(X)$ we get $\kk(X)^{\overline{S}} = \kk(X)^S$.
\end{corollary}

\begin{proof}
	Let $\pi' \colon X \dashrightarrow Y'$
	be a dominant rational map with $\kk(Y) = \kk(X)^S$ (the existence follows from
	the fact, that $ \kk \subseteq \kk(X)^S$ is a finitely generated field extension,
	see e.g.~\cite[Theorem~24.9]{Is2009Algebra:-a-graduat}). After shrinking $X$, 
	we may assume that 
	$\pi' \colon X \to Y'$ is a morphism. 
	Note that $S$ is contained in $\Bir(X/Y')$ and that $\kk(X)^{\Bir(X/Y')} = \kk(Y)$.
	The statement follows now from the fact that $\Bir(X/Y')$ is closed in $\Bir(X)$.
\end{proof}

We now formulate our main result of this section,
which enables us to push forward 
algebraic families in $\Bir(X, \pi\fib)$ parametrized by normal varieties 
to algebraic families in $\Bir(Y)$:

\begin{proposition}
	\label{Prop.Birational_maps_preserving_a_fibration}
	Assume that $\pi$ satisfies Property~\hyperlink{Eq.NiceFibration}{($\ast$)}.
	\begin{enumerate}[left=0pt]
	\item  \label{Prop.Birational_maps_preserving_a_fibration1} 
	We have $\Bir(X, \pi\fib) = \Bir(X, \pi)$. 
	\item  \label{Prop.Birational_maps_preserving_a_fibration2} 
	The natural group homomorphism
	$\Bir(X, \pi) \to \Bir(Y)$ is continuous and 
	preserves algebraic families parametrized by normal varieties.
	\end{enumerate}
\end{proposition}

For the proof of Proposition~\ref{Prop.Birational_maps_preserving_a_fibration} 
we need the following two ingredients:

\begin{lemma}
	\label{Lem.Abstract_bijection_isomorphism}
	Assume that $V, Z$  are irreducible varieties and that $V$ is normal. 
	Then every continuous map $\varphi \colon V \to Z$ that is also a rational map is a morphism.
\end{lemma}

\begin{proof}
	Since $\varphi$ is continuous, its graph
	$\Gamma_{\varphi}$ is closed in $V \times Z$. As $\varphi$ is rational,
	one can take an open, dense subset $U \subseteq V$ such that
	$\varphi' \coloneq \varphi |_U \colon U \to Z$ is a morphism. The graph 
	$\Gamma_{\varphi'}$ is then an open dense subset of $\Gamma_{\varphi}$.
	The projection $\Gamma_{\varphi} \to V$ is a homeomorphism that restricts to 
	an isomorphism $\Gamma_{\varphi'} \xrightarrow{\sim} U$. 
	Since $V$ is normal the projection $\Gamma_{\varphi} \to V$ is an isomorphism, by Zariski's main
	theorem, and hence,
	$\varphi$ is a morphism.
\end{proof}

\begin{lemma}
	\label{Lem.HPrational}
	Assume that $\pi\colon X\to Y$ satisfies Property~\hyperlink{Eq.NiceFibration}{($\ast$)}.
	Let $g \colon Y \to W$ be an abstract map to a variety $W$. If
	$g \circ \pi \colon X \to W$ is rational, then $g$ is rational.
\end{lemma}

\begin{proof}
	Using Lemma~\ref{Lem.NiceFibration_permanence}\eqref{Lem.NiceFibration_permanence1} we first perform several reduction steps.
	Let $\overline{\pi} \colon \overline{X} \to \overline{Y}$ be a 
	completion of $\pi$. We may assume that $W$ is proper and hence after
	resolving the indeterminacy of the rational map
	$g \circ \pi \colon \overline{X} \dashrightarrow W$ we may assume that 
	it is a morphism.
	After replacing $\pi \colon X \to Y$ by the restriction of 
	$\overline{\pi}$ over $Y$ we may assume that $\pi$ is proper
	and $g \circ \pi \colon X \to W$ is a morphism.
	Moreover, we may assume that $X$ is normal after precomposing $\pi$ with the
	normalization of $X$. Let 
	\[
		X \xrightarrow{f} Y' \xrightarrow{\varepsilon} Y
	\]
	be the Stein-factorization of $\pi$. By Lemma~\ref{Lem.Stein_factorization}, $\varepsilon$
	is generically \'etale. After
	restricting $\pi$ to the inverse image of an open dense subset of $Y$,
	we may assume that $Y$ is smooth, the finite 
	morphism $\varepsilon$ is \'etale (and hence $Y'$ is smooth) and $f \colon X \to Y'$
	is flat and surjective. As $\pi = \varepsilon \circ f$ is flat and surjective, $g \colon Y \to W$
	is continuous.
	Let $W_0 \subseteq W$ be an open affine subset that intersects the
	image of $g \colon Y \to W$ densely and consider the open dense
	subsets
	$Y_0 \coloneqq g^{-1}(W_0)$, $Y_0' \coloneqq \varepsilon^{-1}(Y_0)$,
	$X_0 \coloneqq f^{-1}(Y'_0)$ of $Y, Y', W$, respectively.
	Moreover, we consider $W_0$ as a closed subset of $\AA^n$.
	
	As $(g \circ \pi) |_{X_0} \colon X_0 \to W_0$ 
	is a morphism, there exist 
	$h_1, \ldots, h_n \in \OO_X(X_0)$ such that 
	$(g \circ \pi)(x) = (h_1(x), \ldots, h_n(x))$ for all $x \in X_0$. 
	As $f_\ast \OO_X = \OO_{Y'}$,
	we get $\OO_X(X_0) = \OO_{Y_0'}(Y')$ and hence, 
	$(g \circ \varepsilon) |_{Y_0'} \colon Y'_0 \to W_0$ is a morphism.

	Let $\Gamma' \subseteq Y_0' \times W_0$ and $\Gamma \subseteq Y_0 \times W_0$
	be the graphs of $(g \circ \varepsilon) |_{Y_0'}$ and $g |_{Y_0}$, respectively. Both graphs are closed, since 
	both maps are continuous.
	Note that we have the following commutative diagram
	\[
		\xymatrix@R=15pt{
			\Gamma' \ar[d]_-{\textrm{isom.}} \ar[rrr]^-{(y', w) \mapsto (\varepsilon(y'), w)} &&& \Gamma \ar[d]^-{\textrm{homeo.}} \\
			Y'_0 \ar[rrr]^-{\varepsilon}_-{\textrm{\'etale}} &&& Y_0
		}
	\]
	where the vertical arrows are the natural projections. 
	Let $\Gamma_0 \subseteq \Gamma$ be the open dense subset of smooth points of $\Gamma$.
	Hence, the restriction
	of $\Gamma \to Y$ to $\Gamma_0$ 
	yields an injective dominant morphism 
	$\Gamma_0 \to Y$ of smooth varieties with surjective differentials, 
	i.e.~$\Gamma_0 \to Y$ is injective and \'etale. 
	Using that \'etale morphisms are locally standard 
	(see e.g.~\cite[Lemma 29.36.15]{St2024The-Stacks-project}), we conclude that the 
	latter map is an open embedding. This implies that $g \colon Y \to W$ is rational.
\end{proof}

\begin{proof}[{Proof of Proposition~\ref{Prop.Birational_maps_preserving_a_fibration}}]
	Let $\theta$ be an algebraic family of birational transformations of $X$ parametrized 
	by some normal variety $V$ such that $\rho_{\theta}(v) \in \Bir(X, \pi\fib)$ for all $v \in V$.
	We show that there exists a unique algebraic family $\bar{\theta}$ of birational transformations
	of $Y$ parametrized by $V$ such that
	$(\id_V \times \pi) \circ \theta = \bar{\theta} \circ (\id_V \times \pi)$.
	This will then imply the proposition: Indeed, for~\eqref{Prop.Birational_maps_preserving_a_fibration1} we consider the case where $V$ is a point
	and for~\eqref{Prop.Birational_maps_preserving_a_fibration2} we note that for checking 
	closedness of a subset of $\Bir(X)$ it is enough to consider only morphisms from normal varieties.

	\medskip

	We may assume that $V$ is irreducible,
	since the irreducible components of $V$ are pairwise disjoint by the normality of $V$. 
	Using Lemma~\ref{Lem.NiceFibration_permanence}\eqref{Lem.NiceFibration_permanence2}
	it follows that $\id_V \times \pi \colon V \times X \to V \times Y$ 
	satisfies~\hyperlink{Eq.NiceFibration}{($\ast$)}.

	\begin{claim}
		\label{Claim.Settheoretic_map}
		There exist open dense subsets $U_1, U_2 \subseteq V \times Y$ and a bijection $\rho \colon U_1 \to U_2$
		of the closed points such that the following diagram commutes,
		\begin{equation} \label{Eq.Diagram}
			\begin{gathered}
			\xymatrix@R=15pt{
				\lociso(\theta) \ar[d]_-{\eta_1} \ar@{}[r]|-{\supseteq}
				& \eta_1^{-1}(U_1) \ar[d] \ar[r]^-{\theta}_-{\sim} & \eta_2^{-1}(U_2) \ar[d]
				\ar@{}[r]|-{\subseteq} & \lociso(\theta^{-1})  \ar[d]^-{\eta_2} \\
				V \times Y \ar@{}[r]|-{\supseteq} & U_1 \ar[r]^-{\rho} & 
				U_2  \ar@{}[r]|-{\subseteq} & V \times Y
			}	
			\end{gathered}
		\end{equation}
		 where $\eta_1, \eta_2$ are the restrictions of $\id_V \times \pi$.
	\end{claim}

	\begin{proof}
		Since for all $v \in V$ the birational transformation 
		$\rho_{\theta}(v) \in \Bir(X)$ preserves general fibres of $\pi$,
		it follows that $\theta$ maps the fibres of $\eta_1$ isomorphically onto the 
		fibres of $\eta_2$.
		For $i =1, 2$, choose an open dense subset $U_i \subseteq V \times Y$ such that 
		$\eta_i^{-1}(U_i) \to U_i$ is flat and 
		$\theta(\eta_1^{-1}(U_1)) = \eta_2^{-1}(U_2)$. Hence, there exists
		a bijection $\rho \colon U_1 \to U_2$ that makes the diagram~\eqref{Eq.Diagram} 
		commutative. 
	\end{proof}

	Let $Y_0 \subseteq Y$, $X_0 \subseteq X$ be open dense subvarieties with 
	$\pi(X_0) = Y_0$ such that
	$\pi_0 \coloneqq \pi |_{X_0} \colon X_0 \to Y_0$ is flat.
	Denote by $\theta_0$ the restriction of $\theta$ to $V \times X_0$.
	Consider the open dense  
	subsets $W_1 \coloneqq (\id_V \times \pi_0)(\lociso(\theta_0))$ and 
	$W_2 \coloneqq (\id_V \times \pi_0)(\lociso(\theta_0^{-1}))$ of $V \times Y$.
	Since $\theta_0$ preserves the general fibres of $\id_V \times \pi_0$, there is 
	a bijection $\vartheta \colon W_1 \to W_2$
	such that the diagram 
	\[
		\xymatrix@R=15pt{
			\lociso(\theta_0) \ar@{->>}[d]_-{\id_V \times \pi_0} \ar[r]^-{\theta_0}_-{\sim} & \lociso(\theta_0^{-1})
			\ar@{->>}[d]^-{\id_V \times \pi_0} \\
			W_1 \ar[r]^-{\vartheta}_{\textrm{bij.}} & W_{2} \\
		}	
	\]
	commutes. 	Using that $\id_V \times \pi$ satisfies~\hyperlink{Eq.NiceFibration}{($\ast$)}
	we deduce from Claim~\ref{Claim.Settheoretic_map} and Lemma~\ref{Lem.HPrational}
	that $\vartheta$ is birational. 
	Since $X$ and $V$ are normal and $\pi_0$ is flat, it follows that $W_1$ and $W_2$
	are normal (see e.g.~\cite[Corollary~23.9]{Ma1986Commutative-ring-t}) and 
	$\vartheta$ is a homeomorphism. By Lemma~\ref{Lem.Abstract_bijection_isomorphism} 
	we conclude that $\vartheta$
	is in fact an isomorphism. As $W_1, W_2$ project surjectively onto $V$, we get 
	our desired algebraic family.
\end{proof}

When studying $\Bir(X/Y)$ it is sometimes useful to look at the induced birational maps
on the geometric generic fibre.
Let $K$ be an algebraic closure of $\kk(Y)$. We 
denote by $X_K$ the geometric generic fibre of $\pi \colon X \to Y$, which is 
the pull-back of $\pi$ via $\Spec(K) \to Y$. Assume that $X_K$ 
is an irreducible $K$-variety (this is e.g. the case if $\car(\kk) = 0$ and
$\kk(Y)$ is algebraically closed in $\kk(X)$, 
see~\cite[Proposition~5.51]{GoWe0Algebraic-geometryI}).
Now, every $\varphi \in \Bir(X/Y)$
pulls back to a birational map $\varphi_K \in \Bir_K(X_K)$, and we get
a natural injective group homomorphism
\begin{equation}
	\label{Eq.Pullback}
	\varepsilon \colon \Bir(X/Y) \to \Bir_K(X_K) \, , \quad 
	\varphi \mapsto \varphi_K \, .
\end{equation}

\begin{proposition}
	\label{Prop.Pull-back_continuous}
 	If the geometric generic fibre $X_K$ of $\pi \colon X \to Y$ is
	an irreducible $K$-variety, then 
	the injective group homomorphism in~\eqref{Eq.Pullback} is continuous.
\end{proposition}

For the proof of 
Proposition~\ref{Prop.Pull-back_continuous} 
we need to pull back algebraic families. Let 
$\theta \colon V \times X \dashrightarrow V \times X$ be an algebraic family
of birational transformations of $X$ parametrized by some variety $V$ such that
$(\pi \circ \pr_X) \circ \theta = \pi \circ \pr_X$, where $\pr_X$ denotes the projection
to $X$. Then we get via pull-back
a birational transformation
\begin{equation}
	\label{Eq.Pullback_family}
	(V \times K) \times_K X_K =  (V \times X)_K \stackrel{\theta_K}{\dashrightarrow} 
	(V \times X)_K = (V \times K) \times_K X_K \, ,
\end{equation}
where $V \times K$ denotes the fibre product 
$V \times_{\Spec(\kk)} \Spec(K)$.
In general, $\theta_K$ is not an algebraic family parametrized by $V \times K$, 
as $\lociso(\theta_K)$
does not surject onto $V \times K$, see Example~\ref{Exa.Lociso_too_small}. 
However, $\theta_K$ is an algebraic family
parametrized by an open dense subset of $V \times K$ that contains all $\kk$-rational
points of $V$, see Lemma~\ref{Lem.Lociso_of_pull_back_family}.

\begin{example}
	\label{Exa.Lociso_too_small}
	Let $V = \AA^1$, $X = \AA^1 \times \PP^1$, $Y = \AA^1$ and $\pi \colon X \to Y$
	the projection to the first factor and let $K$ be an algebraic closure of 
	$\kk(Y) = \kk(u)$.
	Consider the algebraic family
	\[
		\theta \colon \AA^1 \times (\AA^1 \times \PP^1) \dashrightarrow
		\AA^1 \times (\AA^1 \times \PP^1) \, , \quad
		(t, u, [x:y]) \dashmapsto (t, u, [(tu-1)x:y]) \, .
	\]
	Then, the pull-back is given by
	\[
		\theta_K \colon \AA^1_K \times_K \PP^1_K \dashrightarrow 
		\AA^1_K \times_K \PP^1_K \, , \quad
		(t, [x:y]) \dashmapsto (t, [(tu-1)x:y]) \, .
	\]
	Now, $\lociso(\theta_K) = \lociso(\theta_K^{-1}) = 
	(\AA^1_K \setminus \{ u^{-1} \}) \times_K \PP^1_K$ does not surject onto $\AA^1_K$.
\end{example}

\begin{lemma}
	\label{Lem.Map-k-rational-K-rational_points}
	With the notation introduced above we have:
	\begin{enumerate}[left=0pt]
		\item \label{Lem.Map-k-rational-K-rational_points1} 
		The natural map $\eta \colon V(\kk) \to (V \times K)(K)$ from the $\kk$-rational
		points of $V$ to the $K$-rational points of $V \times K$ is 
		a homeomorphism onto its image.
		\item \label{Lem.Map-k-rational-K-rational_points2} 
		If $U \subseteq V \times Y$ is an open subset that surjects onto $V$,
		then the preimage of $U$ under $V \times K \to V \times Y$
		contains the image of $\eta$, i.e.~the $\kk$-rational points of $V$.
	\end{enumerate}
\end{lemma}

\begin{proof}
	For the proof of both statements we may assume that $V$ is affine. 

	\eqref{Lem.Map-k-rational-K-rational_points1}:
	We denote by $K[V]$ the $K$-valued functions on $V$, i.e. the coordinate ring
	of $V \times K$ over $K$. The map $\eta$ is given by
	\[
		\begin{array}{rcl}
			\{ \, \textrm{maximal ideals in $\kk[V]$} \, \} & \to & 
			\{ \, \textrm{maximal ideals in $K[V] = \kk[V] \otimes_{\kk} K$} \, \} \\
			 \mathfrak{m} & \mapsto & \mathfrak{m} K  \, .
		\end{array}
	\]
	First, we show that $\eta$ is continuous.
	Let $Z$ be a closed subset of $V \times K$ (defined over $K$) and let 
	$I \subseteq K[V]$ be its vanishing ideal. Note that for 
	$\mathfrak{m} \in V(\kk)$ we have
	\[
		\mathfrak{m} \in \eta^{-1}(Z) \iff \mathfrak{m} K \supseteq I 
		\iff \mathfrak{m} K \supseteq J \coloneqq 
		\sum_{\sigma \in \Aut(K / \kk)} \sigma(I)  \, .
	\]
	Thus, the preimages under $\eta$ of the vanishing sets of $I$ and $J$ in 
	$V \times K$ coincide.
	Therefore, we may replace $I$ by $J$ and may assume that $I$ is invariant under $\Aut(K / \kk)$.
	Let $\nu \colon \kk^{[m]} \to \kk[V]$ be a $\kk$-algebra surjection, where 
	$\kk^{[m]}$ denotes the polynomial ring over $\kk$ in $m$ variables. Then $\nu$
	extends to a $K$-algebra surjection 
	$\nu_K \colon K^{[m]} \to K[V]$
	and $\nu_K^{-1}(I)$ is an $\Aut(K / \kk)$-invariant ideal
	of $K^{[m]}$. By \cite[Chp. I, \S7, Lemma~2]{We1946Foundations-of-Alg}
	we have 
	\[
		(\kk^{[m]} \cap \nu_K^{-1}(I))K = \nu_K^{-1}(I) \, .
	\] 
	After applying $\nu_K$ we get thus
	$\nu(\kk^{[m]} \cap \nu_K^{-1}(I))K = I$. Note that
	$\nu(\kk^{[m]} \cap \nu_K^{-1}(I)) = \kk[V] \cap I$
	and therefore $(\kk[V] \cap I)K = I$.
	This shows that
	\[
		\mathfrak{m} \in \eta^{-1}(Z) \iff \mathfrak{m} K \supseteq I 
		\iff \mathfrak{m} \supseteq \kk[V] \cap I
	\]
	and hence the continuity of $\eta$ follows.

	Now it is enough to note that the natural morphism $V \times K \to V$ restricted to
	$\eta(V(\kk))$ yields a continuous map $\omega \colon \eta(V(\kk)) \to V(\kk)$
	such that $\omega \circ \eta = \id_{V(\kk)}$.

	\eqref{Lem.Map-k-rational-K-rational_points2}:  
	Let $W \subseteq V \times Y$ be the preimage of $U$ under 
	$\xi \colon V \times K \to V \times Y$ and let $v \in V$ be a $\kk$-rational point.
	Then $\xi$ restricts to a dominant morphism 
	$\xi^{-1}(\{v\} \times Y) \to \{v\} \times Y$
	of schemes. By assumption $U \cap (\{v\} \times Y)$ is open and dense
	in $\{v\} \times Y$ and therefore $W$ contains the point 
	$\xi^{-1}(\{v\} \times Y)$, which corresponds to $\eta(v)$. 
\end{proof}

\begin{lemma}
	\label{Lem.Lociso_of_pull_back_family}
	The algebraic family $\theta_K$ of~\eqref{Eq.Pullback_family} is parametrized
	by an open dense subset of $V \times K$ that contains all 
	$\kk$-rational points of $V$.
\end{lemma}

\begin{proof}
	After shrinking $X$ and $Y$, we may assume that $\pi \colon X \to Y$ is flat.
	Since $(\id_V \times \pi) \circ \theta = \id_V \times \pi$, it follows that
	the image of $\lociso(\theta)$ and $\lociso(\theta^{-1})$ under 
	$\id_V \times \pi$ coincide. Let us denote this set by $U \subseteq V \times Y$.
	By assumption, $U$ surjects onto $V$. Note that $\theta_K$ restricts
	to an isomorphism $\lociso(\theta)_K \to \lociso(\theta^{-1})_K$
	and the subsets $\lociso(\theta)_K$ and $\lociso(\theta^{-1})_K$ surject
	onto $U_K \subseteq (V \times Y)_K = V \times K$, since pull-backs of 
	surjections are again surjections, see~\cite[Lemma~29.9.4]{St2024The-Stacks-project}.
	Now, the statement follows from 
	Lemma~\ref{Lem.Map-k-rational-K-rational_points}\eqref{Lem.Map-k-rational-K-rational_points2}, since $U_K$ is the preimage of
	$U$ under $V \times K \to V \times Y$.
\end{proof}

\begin{proof}[Proof of Proposition~\ref{Prop.Pull-back_continuous}]
	For proving the continuity of $\varepsilon$, we take a closed 
	$F \subseteq \Bir_K(X_K)$ 
	and an algebraic family $\theta$ of birational transformations of $X$
	parametrized by a variety $V$. We have to show that
	\begin{equation}
		\label{Eq.Subset_V}
		\set{v \in V}{ \rho_{\theta}(v)_K \in F}
	\end{equation}
	is a closed subset of $V$. Denote by $\eta \colon V(\kk) \to (V \times K)(K)$ 
	the natural inclusion.
	Let $W \subseteq V \times K$ be the open subset which is 
	the parameter variety of $\theta_K$.
	By Lemma~\ref{Lem.Lociso_of_pull_back_family}, $W$ contains all $\kk$-rational
	points of $V$.
	Thus, the subset of $V$ in~\eqref{Eq.Subset_V} is equal to 
	\[
		\set{v \in V}{ \eta(v) \in \rho_{\theta_K}^{-1}(F)}	\, ,
	\]
	since $\rho_{\theta}(v)_K = \rho_{\theta_K}(\eta(v))$ for all $v \in V$.
	As $\rho_{\theta_K}^{-1}(F)$ is closed in $W(K)$, this follows from the continuity of
	$\eta \colon V(\kk) \to W(K)$, see Lemma~\ref{Lem.Map-k-rational-K-rational_points}\eqref{Lem.Map-k-rational-K-rational_points1}. 
\end{proof}

\par\bigskip
\renewcommand{\MR}[1]{}
\bibliographystyle{amsalpha}
\bibliography{newBIB}

\end{document}